\documentclass{amsart}

\usepackage{amsmath}
\usepackage{amsthm}
\usepackage{amssymb}
\usepackage{comment}
\usepackage{ulem}
\usepackage{tikz}
\usetikzlibrary{arrows,decorations.pathmorphing,backgrounds,positioning,fit,petri,patterns}

\usepackage{here}

\usepackage{enumitem}

\newcounter{nombre}

\usepackage[nobysame,alphabetic]{amsrefs}
\usepackage{graphicx}
\usepackage{stmaryrd}
\usepackage[all]{xy}
\usepackage[vcentermath]{youngtab}
\usepackage{slashbox}

\allowdisplaybreaks[4]

\newtheorem{defi}{Definition}
\newtheorem{defi-prop}[defi]{Definition-Proposition}
\newtheorem{lemm}[defi]{Lemma}
\newtheorem{prop}[defi]{Proposition}
\newtheorem{theo}[defi]{Theorem}
\newtheorem{coro}[defi]{Corollary}
\newtheorem{conj}[defi]{Conjecture}

\newcommand{\kks}[2][k]{g^{(#1)}_{#2}}

\newcommand{\ks}[1]{s^{(k)}_{#1}}
\newcommand{\Pk}{\mathcal{P}_{k}}
\newcommand{\Cn}{\mathcal{C}_{k+1}}
\newcommand{\tSn}{\ti S_{k+1}/S_{k+1}}
\newcommand{\Lk}{\Lambda^{(k)}}
\newcommand{\la}{\lambda}
\newcommand{\La}{\Lambda}
\newcommand{\ka}{\kappa}
\newcommand{\ga}{\gamma}
\newcommand{\de}{\delta}
\newcommand{\DE}[1]{\de\left[#1\right]}

\newcommand{\Z}{\mathbb{Z}}

\newcommand{\al}[1]{\alpha_{P}(#1)}
\newcommand{\Rl}{R_{k+1-l}}
\newcommand{\Rbl}{R'_{\bl}}
\newcommand{\lra}{\longrightarrow}

\newcommand{\Lra}{\Longrightarrow}

\newcommand{\lmt}{\longmapsto}
\newcommand{\bdd}{\mathfrak{p}}
\newcommand{\core}{\mathfrak{c}}
\newcommand{\StoC}{\mathfrak{s}}

\newcommand{\sm}{\setminus}
\newcommand{\kconj}[1]{{#1}^{\omega_k}}
\newcommand{\ti}{\tilde}

\newcommand{\hook}[2]{\mathrm{hook}_{#1}(#2)}
\newcommand{\wcover}{\prec\hspace{-1.5mm}\cdot\hspace{1mm}}
\newcommand{\ksize}[1]{|#1|_{k+1}}
\newcommand{\Res}{\mathrm{Res}}
\newcommand{\res}{\mathrm{res}}

\newcommand{\syou}[1]{\scalebox{0.45}{\yng(#1)}}

\newcommand{\kl}[1][k]{\kks[#1]{\la}}
\newcommand{\krt}[1][k]{\kks[#1]{R_t}}
\newcommand{\kRt}{\krt}
\newcommand{\krtl}[1][k]{\kks[#1]{R_t\cup\la}}
\newcommand{\kRtl}{\krtl}

\newcommand{\RRm}{R_{t_1}\cup\dots\cup R_{t_m}}
\newcommand{\RRma}{R_{t_1}^{a_1}\cup\dots\cup R_{t_m}^{a_m}}
\newcommand{\DS}{\displaystyle}
\renewcommand{\emptyset}{\varnothing}
\newcommand{\bla}{{\bar\la}}
\newcommand{\bl}{{\bar{l}}}
\newcommand{\NP}{$(\mathsf{NP})$}
\newcommand{\Nla}{$\mathsf{(N\la)}$}

\newcommand{\minindex}{\mathop{\mathrm{minindex}}}

\newcommand{\maxti}{\max_i (t_i)}

\newcommand{\sk}[2]{
	\begin{tikzpicture}[scale=0.18]
		\foreach \li [count=\i] in {#2}{
			\foreach \j in {1,...,\li}{
				\fill[gray] (\j,\i) rectangle +(-1,-1);
			}
		}
		\foreach \li [count=\i] in {#1}{
			\foreach \j in {1,...,\li}{
				\draw (\j,\i) rectangle +(-1,-1);
			}
		}
	\end{tikzpicture}
}

\title{Factorization formulas of $K$-$k$-Schur functions I}
\author{Motoki TAKIGIKU}
\thanks{takigiku@ms.u-tokyo.ac.jp}
\address{Graduate School of Mathematical Sciences, the University of Tokyo, Japan}
\date{\today}

\begin{document}

\maketitle

\begin{abstract}
	We give some new formulas about factorizations of $K$-$k$-Schur functions $\kl$,
analogous to
the $k$-rectangle factorization formula
$\ks{R_t\cup\la}=\ks{R_t}\ks{\la}$
of $k$-Schur functions,
where $\la$ is any $k$-bounded partition and $R_t$
denotes the partition $(t^{k+1-t})$ called \textit{$k$-rectangle}.
Although a formula of the same form does not hold for $K$-$k$-Schur functions,
we can prove that $\krt$ divides $\krtl$, and in fact more generally that
$\kks{P}$ divides $\kks{P\cup\la}$ for any multiple $k$-rectangles $P=\RRma$
and any $k$-bounded partition $\la$.
We give the factorization formula of such $\kks{P}$ and
the explicit formulas of $\kks{P\cup\la}/\kks{P}$ in some cases, 
including the case where $\la$ is a partition with a single part as the easiest example.
\end{abstract}

\tableofcontents

\section{Introduction}

Let $k$ be a positive integer.
\textit{$K$-$k$-Schur functions} $\kks{\la}$
are inhomogeneous symmetric functions parametrized by
$k$-bounded partitions $\la$, 
namely by the weakly decreasing strictly positive integer sequences
$\la=(\la_1,\dots,\la_l)$, $l\in\Z_{\ge 0}$, whose terms are all bounded by $k$.
They are $K$-theoretic analogues of another family of symmetric functions called \textit{$k$-Schur functions}, which are homogeneous and also parametrized by $k$-bounded partitions.

Historically, 
$k$-Schur functions were first introduced by Lascoux, Lapointe and Morse \cite{MR1950481},
and subsequent studies led to several (conjectually equivalent) characterizations of $\ks{\la}$
such as 
the Pieri-like formula due to Lapointe and Morse \cite{MR2331242},
and Lam proved that $k$-Schur functions correspond to the Schubert basis of homology of the affine Grassmannian \cite{Lam08}.
Moreover it was shown by Lam and Shimozono that
$k$-Schur functions play a central role in the explicit description of the Peterson isomorphism
between quantum cohomology of the Grassmannian and homology of the affine Grassmannian
up to suitable localizations
\cite{MR2923177}.

These developments have analogues in $K$-theory.
Lam, Schilling and Shimozono \cite{MR2660675} characterized the $K$-theoretic 
$k$-Schur functions as the Schubert basis of the $K$-homology of the affine
Grassmannian, and Morse \cite{Morse12} investigated them from a conbinatorial
viewpoint, giving various properties including the Pieri-like
formulas using affine set-valued strips (the form using cyclically
decreasing words are also given in \cite{MR2660675}).

In this paper 
we start from this combinatorial characterization (see Definition \ref{KkSchur_def})
and show certain new factorization formulas of $K$-$k$-Schur functions.

Among the $k$-bounded partitions, 
those of the form $(t^{k+1-t})=(\underbrace{t,\dots,t}_{k+1-t})=:R_t$, $1\le t \le k$,
called {\it $k$-rectangle}, play a special role.
In particular, 
if a $k$-bounded partition has the form $R_t\cup \la$,
where the symbol $\cup$ denotes the operation of concatenating the two sequences and reordering the terms in the weakly decreasing order,
then the corresponding $k$-Schur function has the following factorization property \cite[Theorem 40]{MR2331242}:
\begin{equation}\label{intro:zero}
	\ks{R_t\cup\la} = \ks{R_t} \ks{\la}.
\end{equation}
Note that, under the bijection between the set of $k$-bounded partitions and the set of affine Grassmannian elements in the affine symmetric group,
the correspondent of the $k$-rectangle $R_i$ is congruent, 
in the extended affine Weyl group, 
to the translation $t_{-\varpi_i^\vee}$ 
by the negative of a fundamental coweight,
modulo left multiplication by the length zero elements.

It is suggested in \cite[Remark 7.4]{MR2660675} that 
the $K$-$k$-Schur functions should also possess similar properties,
including the divisibility of $\kks{R_t\cup\la}$ by $\kks{R_t}$. 
The present work is an attempt to materialize this suggestion.

We do show in Proposition \ref{P_factor} that 
$\kks{R_t}$ divides $\kks{R_t\cup\la}$ 
in the ring $\Lk=\mathbb{Z}[h_1,\dots,h_k]$,
where $h_i$ denotes the complete homogeneous symmetric functions of degree $i$,
of which the $K$-$k$-Schur functions form a basis.
However, unlike the case of $k$-Schur functions,
the quotient $\kks{R_t\cup\la}/\kks{R_t}$ is not a single term $\kks{\la}$
but, in general, a linear combination of $K$-$k$-Schur functions with leading term $\kks{\la}$,
namely in which $\kks{\la}$ is the only highest degree term.
Even the simplest case where $\la$ consists of a single part $(r)$,
$1\le r \le k$, displays this phenomenon:
we show in Theorem \ref{R_t_r} that
\begin{equation}\label{intro:one}
\kks{R_t\cup(r)} = 
\begin{cases}
	\kks{R_t}\cdot \kks{(r)} & ({\rm if}\ t<r), \\
	\kks{R_t}\cdot \left(\kks{(r)} + \kks{(r-1)} + \cdots + \kks{\emptyset}\right) & ({\rm if}\ t\ge r)
\end{cases}
\end{equation}
(actually we have $\kks{(s)} = h_s$ for $1 \le s \le k$, and $\kks{\emptyset}=h_0=1$).
So we may ask:

\vspace{2mm}
\noindent{\bf Question 1.}
{\it
	Which $\kks{\mu}$, besides $\kl$, appear in the quotient $\krtl/\krt$?
	With what cofficients?
}

\vspace{2mm}
A $k$-bounded partition can always be written in the form
$R_{t_1}\cup\dots\cup R_{t_m} \cup \la$ with 
$\la$ not having so many repetitions of any part as to form a $k$-rectangle.
In such an expression we temporarily call $\la$ the remainder, although this term is only used in the Introduction.
Proceeding in the direction of Question 1,
one ultimate goal may be to give a factorization formula in
terms of the $k$-rectangles and the remainder.
In the case of $k$-Schur functions, the straightforward factorization in (\ref{intro:zero}) above leads
to the formula $\ks{\RRm\cup\la}=\ks{R_{t_1}}\dots\ks{R_{t_m}}\ks{\la}$.
On the contrary, with $K$-$k$-Schur functions, the simplest case having a multiple $k$-rectangle,
to be shown in 
the author's following paper \cite{Takigiku},
gives
\begin{equation}\label{intro:two}
	\kks{R_t\cup R_t} = \krt \sum_{\la\subset R_t} \kl.
\end{equation}
Hence we cannot expect $\kks{R_t\cup R_t}$ to be divisible by $\krt$ twice.
Instead, upon organizing the part consisting of $k$-rectangles in the form 
$\RRma$ with $t_1 < \dots < t_m$ and $a_i \ge 1$ ($1\le i \le m$),
with $R_t^a = \underbrace{R_t\cup\dots\cup R_t}_{a}$, actually we show in Proposition \ref{P_factor} that

\centerline{$\kks{\RRma\cup\la}$ is divisible by $\kks{\RRma}$,}
which actually holds whether or not $\la$ is the remainder.
Then we can subdivide our goal as follows:

\vspace{2mm}
\noindent{\bf Question $1'$.}
{\it
	Which $\kks{\mu}$, besides $\kl$, appear in the quotient $\kks{P\cup \la}/\kks{P}$
	where $P=\RRma$, and with what coefficients?
}

\vspace{2mm}
\noindent{\bf Question $2$.}
{\it
	How can $\kks{\RRma}$ be factorized?
}

\vspace{2mm}
In this (and author's following) paper, we give a reasonably complete answer to Question 2,
and partial answers to Question $1'$.
For Question 2, we first show in Theorem \ref{dist_krec_factor}
that multiple $k$-rectangles of different sizes entirely split,
namely that we have $\kks{\RRma}=\kks{R_{t_1}^{a_1}}\dots\kks{R_{t_m}^{a_m}}$.
Then in the following paper \cite{Takigiku},
we show that for each $1\le t \le k$ and $a>1$,
we have a nice factorization 
$\displaystyle\kks{R_t^a}=\krt \left(\sum_{\la\subset R_t} \kl\right)^{a-1}$,
generalizing the formula (\ref{intro:two}).
Thus, we have
\[
	\kks{\RRma} = \kks{R_{t_1}} \left(\sum_{\la^{(1)}\subset R_{t_1}} \kks{\la^{(1)}}\right)^{a_1-1}
	\dots
	\kks{R_{t_m}} \left(\sum_{\la^{(m)}\subset R_{t_m}} \kks{\la^{(m)}}\right)^{a_m-1}.
\]

For Question $1'$, unfortunately we cannot give a complete answer yet.
Still we obtain some nice explicit formulas, including the case (\ref{intro:one}).

We first show an auxiliary result that, being given
$P=\RRma$ and putting $Q=\RRm$ without multiplicities,
we have $\kks{P\cup\la}/\kks{P} = \kks{Q\cup\la}/\kks{Q}$ for any $\la$.
Thus we can reduce Question $1'$ to the case where the $k$-rectangles are of all different sizes.

We then derive explicit formulas in some limited cases where,
writing $\la=(\la_1,\dots,\la_l)$ and $\bar\la=(\la_1,\dots,\la_{l-1})$,
the parts of $\la$ except for $\la_l$ are all larger than the widths of the $k$-rectangles, 
and $\bar\la$ is contained in a $k$-rectangle.

An easiest case is where $\la=(r)$ consists of a single part, 
which generalizes the case (\ref{intro:one}).
Namely we show that
if $P=\RRma$ with $t_1<\dots<t_m$ and $a_1,\dots,a_m>0$ and $0\le r \le k$,
then $\kks{P\cup(r)}$ decomposes as $\kks{P}\cdot \sum_{s=0}^{r} \binom{\al{r}+s-1}{s} \kks{(r-s)}$,
where $\al{u} = \#\{i\mid 1\le i \le m,\ t_i\ge u\}$.
Considering the case $m=1$ and $a_1=1$, we obtain the formula (\ref{intro:one}).

Generalizing this case, 
we show in Theorem \ref{P_mu_r} that 
if $P$ and $\al{u}$ are the same as above and
$\la=(\la_1,\dots,\la_l)$ satisfies $\la_{l-1}>t_m$ and $\bar\la=(\la_1,\dots,\la_{l-1})$ is contained in a $k$-rectangle,
then $\kks{P\cup\la}$ decomposes as $\kks{P}\cdot \sum_{s=0}^{\la_l} \binom{\al{\la_l}+s-1}{s} \kks{\bar\la\cup(\la_l-s)}$.
In particular,
if $t_n < \la_l$, the summation on the right-hand side consists of a single term $\kl$.
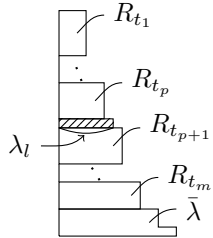
\[
\begin{tikzpicture}[scale=0.12]
	\draw (0,0) -| (13,1) -| (11,3) -| (0,0);
\draw (10,1.5) to [out=45,in=180] (13,3) node [right] {$\bar\la$};

\draw (0,3) rectangle (9,6);
\draw (8,4.5) to [out=45,in=180] (11,6.5) node [right] {$R_{t_m}$};
\draw (0,8) rectangle (7,12);
\draw (6,10) to [out=45,in=180] (9,12) node [right] {$R_{t_{p+1}}$};
\draw (0,13) rectangle (5,17);
\draw (4,15) to [out=45,in=180] (7,17) node [right] {$R_{t_p}$};
\draw (0,20) rectangle (3,25);
\draw (2,22.5) to [out=45,in=180] (5,24.5) node [right] {$R_{t_1}$};
\draw [pattern=north east lines] (0,12) rectangle (6,13);
\draw (0,12) to [out=-20,in=-160] node[inner sep=0pt] (la) {} (6,12);
\draw [<-] (la) to [out=-135,in=0] (-2,10) node [left] {$\la_l$};

\draw (0,6) -- (0,8);
\draw [loosely dotted, thick] (4.5,6.2) -- (3.5,7.8);
\draw (0,17) -- (0,20);
\draw [loosely dotted, thick] (2.5,17.2) -- (1.5,19.8);

\node [right,text width=5cm] at (23,12.5) {
In this figure $p=m-\al{\la_l}$ and $a_i=1$ for all $i$.
};

\end{tikzpicture}
\]

It is worth noting that,
in all cases we have seen, 
$\kks{P\cup\la}\big/\kks{P}$ is a linear combination 
of $K$-$k$-Schur functions 
with {\it positive coefficients}.
Moreover, 
if $P=R_t$,
it seems that
each coefficient is $0$ or $1$ and
the set of $\mu$ such that the coefficient of $\kks{\mu}$ in $\kks{P\cup\la}\big/\kks{P}$ is $1$
is an \textit{interval} (with respect to the strong order. See Conjecture \ref{conj:interval}).
Anyway, it should be interesting to study the geometric meaning of these results and conjectures.

This paper is organized as follows.
In Section \ref{sec:prel},
we review some basic notations and facts about
combinatorial backgrounds of $K$-$k$-Schur functions.
In Section \ref{sec:general},
we show some auxiliary results
which provide a basis for our work.
In Section \ref{sec:one_row},
we give explicit factorization formulas in an easiest case where the remainder consists of a single part.
In Section \ref{sec:dist},
we generalize the result of the previous section and
give a ``straightforward factorization'' formula for 
a multiple $k$-rectangles of different sizes.
In Section \ref{sec:discuss},
we state some observations and conjectures.

\noindent{\bf Acknowledgement. }
The author would like to express his gratitude to
T.\ Ikeda
for suggesting the problem to the
author and helping him with many fruitful discussions.
He is grateful to
H.\ Hosaka and I.\ Terada
for many valuable comments and pointing out
mistakes and typos in the draft version of this paper.
He is also grateful to the committee of
the 29th international conference on
Formal Power Series and Algebraic Combinatorics (FPSAC)
for many valuable comments for the extended abstract version of this paper.
This work was supported by
the Program for Leading Graduate
Schools, MEXT, Japan.
The contents of this paper is the first half of the author's master-thesis \cite{MasterThesis}.

\section{Preliminaries}\label{sec:prel}

In this section we review some 
 requisite combinatorial backgrounds.
%
For detailed definitions,
see for instance \cite[Chapter 2]{MR3379711} or \cite[Chapter I]{MR1354144}.

\subsection{Partitions}\label{prel_partition}

Let $\mathcal{P}$ denote the set of partitions.
A partition $\la=(\la_1 \ge \la_2 \ge \dots)\in\mathcal{P}$ is identified with
its {\it Young diagram} (or {\it shape}), for which we use the French notation here.
quadrant\footnote{We use the French notation} of Cartesian plane so that there are $\la_i$ boxes arranged in left justified way in the $i$-th row from the bottom.
\vspace{-5mm}
\begin{center}
\[
\begin{tikzpicture}[scale=0.25]
	\draw [->] (-0.5,0) -- (6,0);
	\draw [->] (0,-0.5) -- (0,3);
	\draw (0,0) rectangle (1,1);
	\draw (0,1) rectangle (1,2);
	\draw (1,0) rectangle (2,1);
	\draw (1,1) rectangle (2,2);
	\draw (2,0) rectangle (3,1);
	\draw (3,0) rectangle (4,1);
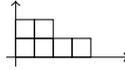
	\node [below] at(3,-1) {the Young diagram of $(4,2)$};
\end{tikzpicture}
\]
\end{center}

\vspace{-1mm}
We denote the \textit{size} of a partition $\la$ by $|\la|$, the \textit{length} by $l(\la)$, and the \textit{conjugate} by $\la'$.
For partitions $\la$, $\mu$ we say $\la\subset\mu$ if $\la_i\le \mu_i$ for all $i$.
%
The {\it dominance order} $\unlhd$ on $\mathcal{P}$ is defined by saying that $\la \unlhd \mu$ if $|\la|=|\mu|$ and $\sum_{i=1}^{r}\la_i\le\sum_{i=1}^{r}\mu_i$ for all $r\ge 1$.
%
%
Sometimes we abbreviate \textit{horizontal strip} (resp.\ \textit{vertical strip}) (of size $r$) to ($r$-)h.s.\ (resp.\ ($r$-)v.s.).
%
For a partition $\la$ and a cell $c=(i,j)$ in $\la$, 
we denote the \textit{hook length} of $c$ in $\la$ by $\hook{c}{\la}=\la_i+\la'_j-i-j+1$.

For a partition $\la$, a {\it removable corner} of $\la$ (or {\it $\la$-removable corner}) is a cell $(i,j)\in\la$ with $(i,j+1),(i+1,j)\notin\la$.
$(i,j)\in(\Z_{>0})^2\sm\la$ is said to be
an {\it addable corner} of $\la$ (or {\it $\la$-addable corner})
if $(i,j-1),(i-1,j)\in\la$
with the understanding that $(0,j),(j,0)\in\la$.
In order to avoid making equations too wide,
we may denote removable corner (resp.\ addable corner) briefly by 
rem.\ cor.\ (resp.\ add.\ cor.).
A cell $(i,j)\in \la$ is called {\it extremal} if $(i+1,j+1)\notin\la$.

For partitions $\la=(\la_1,\ldots,\la_{l(\la)})$, $\mu=(\mu_1,\ldots,\mu_{l(\mu)})$, 
we write $\la \oplus \mu = (\la_1+\mu_1,\la_2+\mu_1,\ldots,\la_{l(\la)}+\mu_1,\mu_1,\ldots,\mu_{l(\mu)})$.
For partitions $\la^{(1)},\ldots,\la^{(n)}$, we define $\la^{(1)}\oplus\cdots\oplus \la^{(n)} = (\la^{(1)}\oplus\cdots\oplus \la^{(n-1)})\oplus \la^{(n)}$, recursively.

\vspace{-2mm}
\[
\begin{tikzpicture}[scale=0.16]
\draw (0,15) -- (1,15) -- (1,14) -- (3,14) -- (3,13) -- (4,13) -- (4,12) -- (0,12) -- cycle;
\draw (2,13) to [out=45,in=180] (6,14.5) node[right] {$\la^{(n)}$};
\draw (4,12) -- (6,12) -- (6,10) -- (8,10) -- (8,8) -- (4,8) -- cycle;
\draw (6,9) to [out=45,in=180] (9,10.5) node[right] {$\la^{(n-1)}$};
\draw (11,5) -- (13,5) -- (13,4) -- (16,4) -- (16,2) -- (17,2) -- (17,0) -- (11,0) -- cycle;
\draw (15,1.5) to [out=45,in=180] (19,3) node[right] {$\la^{(1)}$};
\draw[loosely dotted,very thick] (8,8) -- (11,5);
\draw (0,12) -- (0,0) -- (11,0);
\end{tikzpicture}
\]
\begin{center}
the shape of $\la^{(1)}\oplus\dots\oplus\la^{(n)}$
\end{center}


\subsection{Bounded partitions, cores, affine Grassmannian elements, and $k$-rectangles $R_t$}\label{prel_pca}

A partition $\lambda$ is called {\it $k$-bounded} if $\lambda_1 \le k$.
Let $\Pk$ be the set of all $k$-bounded partitions.
%
An {\it $r$-core} (or simply a \textit{core} if no confusion can arise) is a partition none of whose cells have a hook length equal to $r$.
We denote by $\mathcal{C}_r$ the set of all $r$-core partitions.
When we consider a partition as a core, the notion of size differs from the usual one:
the {\it length} (or {\it size}) of an $r$-core $\ka$ is the number of cells in $\ka$ whose hook length is smaller than $r$, and denoted by $|\ka|_{r}$.

The {\it affine symmetric group} $\ti S_{k+1}$ is given by 
generators $\{ s_0,s_1,\dots,s_{k}\}$
and relations
$s_i^2=1$,
$s_i s_{i+1} s_i = s_{i+1} s_i s_{i+1}$,
$s_i s_j = s_j s_i$ for $i-j \not\equiv 0,1,k \mod (k+1)$,
with all indices are considered mod $(k+1)$.
Note that the symmetric group $S_{k+1}$ generated by $\{s_1,\dots,s_{k}\}$ is a subgroup of $\ti S_{k+1}$.
We identify the left cosets of $\ti S_{k+1}/S_{k+1}$ with their minimal length representatives, which we call {\it affine Grassmannian elements}.
Namely, the set of affine Grassmannian elements is
$\{ w\in \ti S_{k+1} \mid l(w s_i)>l(w)\ (\forall i\neq 0)\}$.

\vspace{2mm}
{\it Hereafter we fix a positive integer $k$.}

For a cell $c=(i,j)$, the {\it content} of $c$ is $j-i$ and
the {\it residue} of $c$ is $\mathrm{res}(c)=j-i \mod (k+1) \in \mathbb{Z}/(k+1)$.
For a set $X$ of cells, we write
$\mathrm{Res}(X) = \{\,\res(c) \mid c\in X\,\}$.
%
We will write a $\la$-removable corner of residue $i$ simply a $\la$-removable $i$-corner.
For simplicity of notation, we may use an integer to represent a residue, omitting ``mod $(k+1)$''.

We denote by $R_t$ the partition 
$(t^{k+1-t})=(t,t,\dots,t) \in \Pk$ for $1\le t \le k$,
which is called a {\it $k$-rectangle}.
Naturally a $k$-rectangle is a $(k+1)$-core.

\vspace{2mm}
Now we recall the
bijection between
the $k$-bounded partitions in $\Pk$, the $(k+1)$-cores in $\Cn$, and the affine Grassmannian elements in $\tSn$:
\[\xymatrix{
		\Pk \ar@<0.5ex>[rr]^{\core} \ar[rd]_{\substack{\text{by taking}\\\text{``word''}}} & & \Cn \ar@<0.5ex>[ll]^{\bdd} \\
     & \tSn \ar[ru]_{\StoC} & 
}
\]

\noindent\underline{\it The maps $\bdd$ and $\core$:}

The map
	$\bdd \colon \Cn \lra \Pk ; \ka \mapsto \la$
is defined by
$\la_i=\#\{j\mid (i,j)\in\ka,\ \hook{(i,j)}{\ka}\le k\}$.

The map
	$\core \colon \Pk \lra \Cn ; \la \mapsto \ka $
is defined by the following procedure:
given a $k$-bounded partition $\la$ then work from the smallest part to the largest.
For each row, calculate the hook lengths of all its cells.
If there is a cell with hook length greater than $k$, slide this row to the right until all its cells have hook length not greater than $k$.
In the end this process produces a skew shape $\mu/\nu$, where in fact $\mu$ is a $(k+1)$-core.
Then let $\ka$ be this $\mu$.

Then in fact $\bdd$ and $\core$ are bijective and $\bdd=\core^{-1}$.
See \cite[Theorem 7]{MR2167475} for the proof.
The next lemma gives 
a more explicit description for $\core$,
which follows from the argument given just before \cite[Example 1.23]{MR3379711}:
\begin{lemm}\label{core_j}
For $\la \in \Pk$ and $j\geq 1$,
$ \core(\la)_j = \core(\la)_{j+k+1-\la_j}+\la_j$. 
\end{lemm}

Note that if $\la$ is contained in a $k$-rectangle then $\la\in\Pk$ and $\la\in\Cn$, and besides $\bdd(\la)=\la=\core(\la)$.

\vspace{2mm}
\noindent\underline{\it The map $\StoC$ and the inverse:}
For $\ka\in\Cn$ and $i=0,1,\dots,k$,
we define $s_i\cdot \ka$ as follows:
\begin{itemize}
	\item if there is a $\ka$-addable $i$-corner, 
		then let $s_i\cdot \ka$ be 
		$\ka$ with all $\ka$-addable $i$-corners added,
	\item if there is a $\ka$-removable $i$-corner, 
		then let $s_i\cdot \ka$ be 
		$\ka$ with all $\ka$-removable $i$-corners removed,
	\item otherwise, let $s_i\cdot \ka$ be $\ka$.
\end{itemize}

In fact first and second case never occur simultaneously and
$s_i\cdot \ka\in\Cn$ and
then we have a well-defined $\ti S_{k+1}$-action on $\Cn$ and
it induces a bijection
$$\StoC : \tSn \longrightarrow \Cn ; w \mapsto w\cdot \emptyset.$$

The inverse map is given by
\[
\Cn\lra\Pk\lra \tSn ; \ka \mapsto \bdd(\ka) \mapsto w_{\bdd(\ka)},
\]
where $w_{\la}$ is the affine permutation
$s_{i_1} s_{i_2}\dots s_{i_l}$, where $(i_1,i_2,\dots,i_l)$ is the sequence
obtained by reading the residues of the cells in $\la$, from the shortest row to the largest, and within each row from right to left.
See \cite[Corollary 48]{MR2167475} for the proof.

\subsection{Weak order and weak strips}\label{prel_wo}

In this subsection we review the weak order on $\Pk\simeq\Cn\simeq\tSn$.

For a $k$-bounded partition $\la$, its {\it $k$-conjugate} $\kconj{\la}$ is also a $k$-bounded partition given by $\kconj{\la}=\bdd(\core(\la)')$.

\begin{defi-prop}\label{weakorder}
The {\it weak order} $\prec$ on $\tSn$ is defined by the following covering relation:
\begin{equation}
	w \wcover v :\iff \text{$\exists i$ such that $s_i w=v$, $l(w)+1=l(v)$.} \label{wo_Sn}
\end{equation}

It is transferred to 
$\Pk$ and $\Cn$
by the bijection described above
 as follows:
\begin{align}
	&\text{on $\Pk$:}&	\la \wcover \mu &\iff \la\subset\mu,\ \kconj{\la}\subset\kconj{\mu},\ |\la|+1=|\mu|. \label{wo_Pk} \\
 &\text{on $\Cn$:}& \tau \wcover \ka &\iff \text{$\exists i$ such that $s_i \tau=\ka$, $|\tau|_{k+1}+1=|\ka|_{k+1}$.} \label{wo_Cn}
\end{align}
\end{defi-prop}
\begin{proof}
	$(\ref{wo_Sn}) \iff (\ref{wo_Cn})$: see \cite{MR1851953}.
	$(\ref{wo_Pk}) \iff (\ref{wo_Cn})$: see \cite[Corollary 25]{MR2167475}.
\end{proof}

\begin{defi-prop}\label{weakstrip}
	For $(k+1)$-cores $\tau\subset\ka\in\Cn$,
	$\ka/\tau$ is called a {\it weak strip} of size $r$
	(or a {\it weak $r$-strip})
	if the following equivalent conditions hold:
\begin{enumerate}
\renewcommand{\labelenumi}{(\arabic{enumi})}
\item $\ka/\tau$ is horizontal strip and 
	$\tau \wcover \exists \tau^{(1)} \wcover \dots \wcover \exists \tau^{(r)} = \ka$.
\item $\ka/\tau$ is horizontal strip and 
	$\ksize{\ka}=\ksize{\tau}+r$ and
	$\#\mathrm{Res}(\ka/\tau)=r$.
\item $\bdd(\ka)/\bdd(\tau)$ is a horizontal strip and 
	$\bdd(\ka')/\bdd(\tau')$ is a vertical strip and
	$\ksize{\ka} = \ksize{\tau}+r$.
\item $\ka = s_{i_1}\dots s_{i_r} \tau$ 
	for some cyclically decreasing element $s_{i_1}\dots s_{i_r}$.\\
	(Here, an affine permutation $w=s_{i_1}\dots s_{i_r}$ 
	(a reduced expression) is called
	{\it cyclically decreasing} if 
	$i_1,\dots,i_r$ are distinct and
	$j$ never precedes $j+1$ (taken modulo $k+1$)
	in the sequence $i_1 i_2 \dots i_r$.
	This definition is in fact independent of which reduced expression we choose.
	)
\end{enumerate}
\end{defi-prop}
\begin{proof}
$(1)\Lra(3)$: see \cite[Theorem 58]{MR2167475}.
$(3)\Lra(1)$: see \cite[Proposition 54, Theorem 56]{MR2167475}.
$(3)\Lra(2)$: see \cite[Theorem 56]{MR2167475}.

$(4)\Lra(1)$, $(1)\Lra(4)$: 
see Appendix \ref{sec:apdx_ws}.

$(2)\Lra(1)$: omitted since (2) is not used in this paper.
\end{proof}

\subsection{Symmetric functions}\label{prel_sym}

Let $\La=\Z[h_1,h_2,\dots]$ be the ring of symmetric functions,
generated by the 
\textit{complete symmetric functions}
$ h_r = \sum_{i_1\le i_2\le\dots\le i_r} x_{i_1}\dots x_{i_r}$.
%
For a partition $\la$
we set 
$h_\la = h_{\la_1}h_{\la_2}\dots h_{\la_{l(\la)}}$.
Then
$\{h_\la\}_{\la\in \mathcal{P}}$
forms a $\Z$-basis of $\La$.
%
The \textit{Schur functions} $\{s_\la\}_{\la\in\mathcal{P}}$ are the family of symmetric functions satisfying the \textit{Pieri rule}:
\[ h_r s_\la = \sum_{\text{$\mu/\la$:horizontal $r$-strip}} s_\mu.\]

Note that
$h_r s_{\la} = s_{\la\cup(r)} + \sum_{\mu\rhd\la\cup(r)} a_{\mu}s_{\mu}$
for some $a_\mu$.
Using this repeatedly, we can write
$h_\la = s_\la + \sum_{\la\lhd\mu} K_{\mu\la} s_{\mu}$
for some coefficients $K_{\mu\la}$.
%
Thus Schur functions $\{s_\la\}_{\la\in\mathcal{P}}$ form a basis of $\La$
since the transformation matrix between $\{s_{\la}\}_{\la}$ and $\{h_{\mu}\}_{\mu}$ is unitriangular.

\subsection{$k$-Schur functions}\label{prel_kS}

We recall a characterization of $k$-Schur functions given in \cite{MR2331242},
since it is a model for and has a relationship with $K$-$k$-Schur functions.

\begin{defi}[$k$-Schur function via ``weak Pieri rule'']
$k$-Schur functions
$\{\ks{\la}\}_{\la\in\Pk}$ are the family of symmetric functions
such that 
\begin{align*}
	\ks{\emptyset}&=1, \\
	h_r \ks{\la} &= \sum_{\mu} \ks{\mu} \quad \text{for $r\le k$ and $\mu\in\Pk$,}
\end{align*}
summed over $\mu\in\Pk$ such that $\core(\mu)/\core(\la)$ is a weak strip of size $r$.
\end{defi}

According to the fact that
if $\core(\nu)/\core(\eta)$ is a weak strip then $\nu/\eta$ is a horizontal strip,
we can write 
$h_\la = \ks{\la} + \sum_{\la\lhd\mu\in\Pk} K^{(k)}_{\mu\la} \ks{\mu}$
for $\la\in\Pk$
by the same argument as the case of Schur functions,
which ensures the well-definedness of $\ks{\la}$
and shows that $\{\ks{\la}\}_{\la\in\Pk}$ forms a basis of $\Lk=\Z[h_1,\dots,h_k]\subset\La$.
In addition $\ks{\la}$ is homogeneous of degree $|\la|$.

Note that $\ks{(r)}=h_r$ for $1\le r \le k$ since $\core(\la)/\emptyset$ is a weak $r$-strip if and only if $\la=(r)$.
In \cite[Property 39]{MR2331242} it is proved that if $\la_1+l(\la)\le k+1$ (in other words $\la\subset R_t$ for some $t$) then $\ks{\la}=s_{\la}$.

It is proved in \cite[Theorem 40]{MR2331242} that
\begin{prop}[$k$-rectangle property]\label{prop:kS_fac}
For $1\le t \le k$ and $\la\in\Pk$, we have
$
	\ks{R_t\cup\la} = \ks{R_t} \ks{\la} (=s_{R_t} \ks{\la}).
$
\end{prop}

\subsection{$K$-$k$-Schur functions $\kl$}\label{prel_KkS}
In \cite{Morse12} a combinatorial characterization of $K$-$k$-Schur functions is given
via an analogue of the Pieri rule, using some kind of strips called \textit{affine set-valued strips}.

\vspace{2mm}
For a partition $\la$, $(i,j)\in(\Z_{>0})^2$ is called {\it $\la$-blocked} if $(i+1,j)\in \la$.

\begin{defi}[affine set-valued strip]\label{asvs}
	For $r\le k$,
	$(\ga/\beta,\rho)$ is called an {\it affine set-valued strip} of size $r$
	(or an {\it affine set-valued $r$-strip})
	if
	$\rho$ is a partition and
	$\beta\subset\ga$ are cores both containing $\rho$ such that
	\begin{enumerate}
			\renewcommand{\labelenumi}{(\arabic{enumi})}
	\item $\ga/\beta$ is a weak $(r-m)$-strip where we put $m=\#\mathrm{Res}(\beta/\rho)$,
	\item $\beta/\rho$ is a subset of $\beta$-removable corners,
	\item $\ga/\rho$ is a horizontal strip,
	\item For all $i\in\Res(\beta/\rho)$, 
		all $\beta$-removable $i$-corners which are not $\ga$-blocked are in $\beta/\rho$.
	\end{enumerate}
\end{defi}

In this paper we employ the following characterization \cite[Theorem 48]{Morse12} of the $K$-$k$-Schur function as its definition.
\begin{defi}[$K$-$k$-Schur function via an ``affine set-valued'' Pieri rule]\label{KkSchur_def}
$K$-$k$-Schur functions $\{\kks{\la}\}_{\la\in\Pk}$ are the family of symmetric functions such that 
$\kks{\emptyset}=1$ and 
for $\la \in \Pk$ and $0 \le r \le k$,
\[
h_r \cdot \kks{\la} =
\sum_{(\mu,\rho)} (-1)^{|\la|+r-|\mu|} \kks{\mu},
	\]
summed over $(\mu,\rho)$ such that $(\core(\mu)/\core(\la),\rho)$ is an affine set-valued strip of size $r$.
\end{defi}

Notice that, 
given a weak strip $\ga/\beta$, taking a $\rho$ such that $(\ga/\beta,\rho)$ becomes an affine set-valued strip is equivalent to choosing a subset of the set of residues $i\in \mathbb{Z}/(k+1)$ where
there is at least one $\ga$-nonblocked $\beta$-removable $i$-corner.

Now we introduce a notation for the convinience:
\begin{defi}
	For partitions $\la,\mu$, 
	we denote by $r_{\la\mu}$ the number of distinct residues of $\la$-nonblocked $\mu$-removable corners.
\end{defi}

Then for a fixed weak $(r-m)$-strip $\ga/\beta$, the number of $\rho$ such that $(\ga/\beta,\rho)$ is an affine set-valued $r$-strip is equal to
$\displaystyle\binom{r_{\ga\beta}}{m}$.
Notice that $\ga/\beta$ with all $\ga$-nonblocked $\beta$-removable corners added is a horizontal strip.
Therefore we can rewrite Definition \ref{KkSchur_def}:

\begin{prop}\label{Pieri}
For $\la \in \Pk$ and $0 \le r \le k$,
\begin{equation}\label{eq:Pieri}
h_r \cdot \kks{\la} = \sum_{s=0}^{r} (-1)^{r-s} 
\sum_{\substack{\mu \\ \core(\mu)/\core(\la):\text{weak $s$-strip}}} \binom{r_{\core(\mu)\core(\la)}}{r-s} \kks{\mu}. 
\end{equation}
\end{prop}

		We can prove similarly that
		$\kks{\la}$ is uniquely determined by (\ref{eq:Pieri}):
		for $\la\in\Pk$ and $1\le r\le k$, we have
		$h_r \kks{\la} = \kks{\la\cup(r)} + \sum_{\mu} a_\mu \kks{\mu}$
		with $\mu\in\Pk$ satisfying $|\mu|<|\la\cup(r)|$ or $\mu \rhd \la\cup(r)$.
		Thus, for $\la\in\Pk$ we can write 
		$h_\la = \kks{\la} + \sum_{\mu} \mathcal{K}^{(k)}_{\mu\la} \kks{\mu}$,
		summed over $\mu\in\Pk$ satifying $|\mu|<|\la|$ or $\mu \rhd \la$.
		Hence $\kks{\la}$ is well-defined and $\{\kks{\la}\}_{\la\in\Pk}$ forms a basis of $\Lk$.

Note that
		$\kks{(r)}=h_r$ for $1\le r \le k$ since if $(\core(\mu)/\emptyset,\rho)$ is an affine set-valued $r$-strip then $(\mu,\rho)=((r),\emptyset)$.
		Moreover,
		though $\kks{\la}$ is an inhomogeneous symmetric function in general,
		from the form of (\ref{eq:Pieri}) we can deduce that the degree of $\kks{\la}$ is $|\la|$
		and its homogeneous part of highest degree is equal to $\ks{\la}$ by using induction.

\subsection{Some properties of bounded partitions and cores}\label{prel_prop}

In this section we review some properties which show that
the $k$-rectangles $R_t=(t^{k+1-t})$ are important
and thus it can be expected that there are some good properties of $\kks{\la}$'s where $\la$ can be written in the form $\la=R_t\cup\mu$.

Recall the weak order $\prec$ of Definition \ref{weakorder}.

\begin{coro}\label{ws_rect_P}
	For $\mu,\la \in \Pk$ and 
	$P=R_{t_1}^{a_1}\cup\cdots\cup R_{t_m}^{a_m}$ 
	$($$1\le t_1< \cdots < t_m \le k$ and $a_1,\dots,a_m\in\Z_{>0}$$)$, 
\[
	\la\cup P \preceq \mu \iff 
	\exists \nu \in \Pk,
	\begin{cases}
		\mu = \nu \cup P, \\
		\la \preceq \nu.
	\end{cases}
\]
\end{coro}
\begin{proof}
	See \cite[Theorem 20]{MR2079931} for
	the case where $m=1$ and $a_1=1$ (i.e.\ $P=R_{t_1}$).
	The general case follows by using this case repeatedly.
\end{proof}

\begin{prop}\label{ws_Rt}
For $\nu,\la \in \Pk$ and 
	$P=R_{t_1}^{a_1}\cup\cdots\cup R_{t_m}^{a_m}$ 
	$($$1\le t_1< \cdots < t_m \le k$ and $a_1,\dots,a_m\in\Z_{>0}$$)$, 
\[
	\text{$\core(\nu)/\core(\la)$ is a weak strip}
	\iff 
	\text{$\core(\nu\cup P) / \core(\la\cup P)$ is a weak strip}
\]
\end{prop}
\begin{proof}
	The case where $m=1$ and $a_1=1$ (i.e.\ $P=R_{t_1}$) is
	proved in the proof of \cite[Theorem 40]{MR2331242}.	
	The general case follows by using this case repeatedly.
\end{proof}

\begin{coro}\label{weaks_rect_P}
	For $\eta,\la \in \Pk$ and 
	$P=R_{t_1}^{a_1}\cup\cdots\cup R_{t_m}^{a_m}$ 
	$($$1\le t_1< \cdots < t_m \le k$ and $a_1,\dots,a_m\in\Z_{>0}$$)$, 
\[
	\text{$\core(\mu)/\core(\la\cup P)$ is a weak strip} \iff 
	\exists \nu \in \Pk,
	\begin{cases}
		\mu = \nu \cup P, \\
		\text{$\core(\nu)/\core(\la)$ is a weak strip}.
	\end{cases}
\]
\end{coro}
\begin{proof}
	$\Lra$:
	We have $\la\cup P\preceq\mu$ by the definition of weak strips.
	Thus we can write $\mu=\exists\nu\cup P$ 
	by Proposition \ref{ws_rect_P}.
	Then we have that $\core(\nu)/\core(\la)$ is a weak strip by Proposition \ref{ws_Rt}.

	$\Longrightarrow$:
	By Proposition \ref{ws_Rt}.
\end{proof}

\begin{lemm}\label{core_tila_i}
Let $\la\in\Pk$, $1\le t\le k$,
and let $r\in\Z_{\ge 0}$ such that
$\la_r \ge t \ge \la_{r+1}$, where we regard $\la_0=\infty$.
Put $\tilde\la = \la \cup R_t$.
Then
\[
\core(\ti\la)_i = 
\begin{cases}
	\core(\la)_i+t & \ (\text{if $i\le r+(k+1-t)$}), \\
	\core(\la)_{i-(k+1-t)} & \ (\text{if $i\ge (r+1)+(k+1-t)$}).
\end{cases}
\]
\end{lemm}
\begin{proof}
The latter
case
is obvious since
$\ti\la_{i+(k+1-t)}=\la_{i}$ for $i\ge r+1$.

For $i = r+(k+1-t),\ldots,r+1$,
\begin{align*}
	\core(\ti\la)_i&=\core(\ti\la)_{i+k+1-\ti\la_i} + \ti\la_i \qquad \text{(by Lemma \ref{core_j})}\\
				   &=\core(\ti\la)_{i+k+1-t}+t \qquad \text{(since $\ti\la_i=t$)} \\
				   &=\core(\la)_{i} + t. \qquad \text{(by the latter case)}
\end{align*}
Then for $i=r,r-1,\ldots,1$, by descending induction on $i$,
\begin{align*}
	\core(\ti\la)_i&= \core(\ti\la)_{i+k+1-\ti\la_i}+\ti\la_i & \text{(by Lemma \ref{core_j})}\\
				   &= \core(\ti\la)_{\scriptsize{\underbrace{i+k+1-\la_i}_{\le r+k+1-t}}} + \la_i &\text{(since $i\le r$)} \\
				   &= \core(\la)_{i+k+1-\la_i} + t + \la_i &\text{(induction hypothesis)} \\
				   &= \core(\la)_i + t. &\text{(by Lemma \ref{core_j})}
\end{align*}
\end{proof}

\noindent{\it Remark.}
There are more than one candidates for $r$
if $\la$ has a part equal to $t$,
thus in such situations both equalities of the above lemma may hold for some $i$.

\section{Possibility of factoring out $\kks{\RRma}$ and some other general results}\label{sec:general}
Recall how to prove the formula $\ks{R_t \cup \la}=\ks{R_t}\ks{\la}$ in \cite{MR2331242}:
first consider a linear map $\Theta$ extending $\ks{\la}\mapsto\ks{R_t\cup\la}$ for all $\la\in\Pk$.
Then from the weak Pieri rule it was shown that it commutes with the multiplication by $h_r$,
and thus that $\Theta$ coincides with the multiplication by $\ks{R_t}$.
In the case of $K$-$k$-Schur functions, a similar map $\Theta$ does not commute with the multiplication of $h_r$ since the Pieri rule is different in lower terms.
However, it holds that $\kks{R_t}$ divides $\kks{R_t \cup \la}$.
We prove it in a slightly more general form.

The following notation is often referred later:
\begin{itemize}
	\item[\NP]\label{nota:NP}
		Let $1\le t_1,\dots,t_m \le k$ be distinct integers
		and $a_i\in\Z_{>0}\ (1\le i\le m)$,
		where $m\in\Z_{>0}$.
		Then we put
		\begin{align*}
			P&=\RRma,	 \\
			\al{u} &= \#\{t_v\mid 1\le v \le m,\ t_v\ge u\} \quad\text{for each $u\in\Z_{>0}$}.
		\end{align*}
\end{itemize}

\begin{prop}\label{P_factor} 
Let $P$ be as in the above \NP.
Then, for $\lambda=(\lambda_1,\cdots,\lambda_l) \in \Pk$, 
we have $\kks{P} | \kks{\lambda \cup P}$ in the ring $\Lk$.
\end{prop}

\noindent\textit{Remark.} 
Note that $\la$ may still have the form $\la=R_t\cup\mu$.
Hereafter we will not repeat the same remark in similar statements.

\begin{proof}
we prove it by induction on $\lambda$,
with respect to
the order $\leq$ defined by
$\mu\leq\la \iff |\mu|<|\la|$ or ($|\mu|=|\la|$ and $\mu\rhd\la$).
The statement is obvious when $\lambda = \emptyset$.

Assume $\la \neq \emptyset$ and
put $\hat{\lambda} = (\lambda_1,\cdots,\lambda_{l-1})$. Then
$$\kks{P\cup\hat\lambda} \cdot \kks{(\lambda_l)} = 
\kks{P\cup\lambda} + \sum_{\mu} a_{\lambda\mu}\kks{P\cup \mu}$$
for some coefficients $a_{\lambda\mu}$,
since adding a weak strip to $P\cup \hat\la$ yields a $k$-bounded partition
in the form of $P\cup \mu$ for some $\mu\in\Pk$,
by Proposition \ref{ws_rect_P}.
Here $\mu$ in the summation runs under the condition
$|\mu| < |\lambda|$ or $\mu \rhd \lambda$.
By induction hypothesis $\kks{P\cup\hat\lambda}$ and $\kks{P\cup\mu}$ are
divisible by $\kks{P}$ if $|\mu|<|\lambda|$ or $\mu \rhd \lambda$.
This completes the proof.
\end{proof}

Since the homogeneous part of highest degree of $\kks{\la}$ is equal to $\ks{\la}$ for any $\la$,
it follows from Propositions \ref{prop:kS_fac} and \ref{P_factor} that
\begin{coro}
	Let $P$ be as in \NP.
	Then, for any $\la\in\Pk$,
	we can write
	$$
	\kks{P \cup \la} 
	= \kks{P} \left(\kks{\la} + \sum_{\mu} a_{\la\mu}\kks{\mu}\right),
	$$
	summing over $\mu\in\Pk$ such that $|\mu|<|\la|$,
	for some coefficients $a_{\la\mu}$ (depending on $P$).
\end{coro}

\vspace{2mm}
Now we are interested in finding a explicit description of $\kks{P\cup\la}\big/\kks{P}$.
Let us consider the case $P=R_t$ for simplicity.


As noted above, 
a linear map $\Theta$ extending $\kl \mapsto \kRtl$ ($\forall\la\in\Pk$) does not
coincide with the multiplication of $\kRt$
because it does not commute with the multiplication by $h_r$ in the first place.

However,
in the remaining part of this section, we can prove that the restriction of $\Theta$ to the subspace spanned by $\{\kks{R_t\cup\mu}\}_{\mu\in\Pk}$ (in fact this is the principal ideal generated by $\krt$)
commutes with the multiplication by $h_r$,
and thus it coincides with the multiplication of $\Theta(\krt)\big/\krt = \kks{R_t\cup R_t}\big/\krt$ on that ideal (Proposition \ref{Rt_twice}).
Thus it is of interest to describe the value of
$\kks{R_t\cup R_t}\big/\krt$,
which is shown to be $\sum_{\nu\subset R_t}\kks{\nu}$
in the author's following paper \cite{Takigiku}.

\vspace{2mm}
Now let us begin with seeing how $\Theta$ and the multiplication by $h_r$ do \textit{not} commute.
Recall the $K$-$k$-Schur version of the Pieri rule (\ref{eq:Pieri})
\begin{equation*}
h_r \cdot \kks{\la} = \sum_{s=0}^{r} (-1)^{r-s} 
\sum_{\substack{\nu \\ \text{$\core(\nu)/\core(\la)$:weak $s$-strip}}} \binom{r_{\core(\nu)\core(\la)}}{r-s} \kks{\nu},
\end{equation*}
and compare with the formula obtained by
replacing $\la$ with $R_t \cup \la$:
\begin{equation}\label{eq:Pieri_Rt_la}
h_r \cdot \kks{R_t\cup\la} = \sum_{s=0}^{r} (-1)^{r-s} 
\sum_{\substack{\eta \\ \text{$\core(\eta)/\core(R_t\cup\la)$:weak $s$-strip}}} \binom{r_{\core(\eta)\core(R_t\cup\la)}}{r-s} \kks{\eta}.
\end{equation}
By Corollary \ref{weaks_rect_P},
the summation in (\ref{eq:Pieri_Rt_la}) is formed for all $\eta$
having the form $\eta=R_t\cup\nu$ such that $\core(\nu)/\core(\la)$ is a weak $s$-strip.
Hence the right-hand side of (\ref{eq:Pieri_Rt_la}) differs from
what is obtained by replacing each $\nu$ in the right-hand side of (\ref{eq:Pieri}) by $R_t\cup\nu$
according to the
difference between 
$r_{\core(\nu)\core(\la)}$ and 
$r_{\core(R_t\cup\nu)\core(R_t\cup\la)}$.

The next lemma says $r_{\core(R_t\cup\nu)\core(R_t\cup\la)}=r_{\core(\nu)\core(\la)}$ holds if $\la$ has a part equal to $t$.

\begin{lemm}\label{p=p}
	For $\nu,\la \in \Pk$ such that
	$\la$ has a part equal to $t$
	and $\core(\nu)/\core(\la)$ is a weak strip, we have
	$r_{\core(\nu),\core(\la)} = r_{\core(\nu\cup R_t),\core(\la\cup R_t)}$.
\end{lemm}

\begin{proof}
	We write
	$\tilde{\la} = \la \cup R_t$ and $\tilde{\nu} = \nu\cup R_t$.
	We take $r$ such that $\la_r = t > \la_{r+1}$\ (then $\nu_r \ge t = \la_r \ge \nu_{r+1}$ since $\nu/\la$ is a horizontal strip).
	Then we have
	\begin{align*}
		\tilde{\la}_r = \tilde{\la}_{r+1} &= \cdots = \tilde{\la}_{r+k+1-t} = t \\
		                \tilde{\nu}_{r+1} &= \cdots = \tilde{\nu}_{r+k+1-t} = t,
	\end{align*}
therefore by Lemma \ref{core_tila_i}
	\begin{align*}
		&\core(\ti\la)_i = \core(\la)_i + t   &\ \text{($i \le r+(k+1-t)$)} \\
		&\core(\ti\la)_{i} = \core(\la)_{i-(k+1-t)}  &\ \text{($i \ge r+(k+1-t)$)}
	\end{align*}
	(here we applied Lemma \ref{core_tila_i} to $\la$ and $r-1$ for the lower equation)
	and
	\begin{align*}
		&\core(\ti\nu)_i = \core(\nu)_i + t   &\ \text{($i \le r+(k+1-t)$)} \\
		&\core(\ti\nu)_{i} = \core(\nu)_{i-(k+1-t)}  &\ \text{($i > r+(k+1-t)$)}.
	\end{align*}
	\[
	\begin{tikzpicture}[scale=0.26]
\draw [red,very thick] (5,3) |- (8,2) |- (12,1) -- (12,0);
\draw [red,very thick] (5,6) |- (8,5) |- (12,4) -- (12,3);
\draw [red,very thick] (18,3) |- (21,2) |- (25,1) -- (25,0);

\draw (0,0) |- (3,4) |- (5,3) |- (8,2) |- (12,1) |- (18,0) |- (26,-1) -- (26,-2);
\draw [loosely dotted,thick] (0,0) -- (0,-2);
\draw [loosely dotted,thick] (26,-2) -- (26,-3);

\draw [pattern=north east lines] (0,4) rectangle (2,5);
\draw [pattern=north east lines] (5,2) rectangle (6,3);
\draw [pattern=north east lines] (8,1) rectangle (10,2);
\draw [pattern=north east lines] (12,0) rectangle (14,1);

\draw [pattern=dots] (-0.1,7) rectangle (2,8);
\draw [pattern=dots] (5,5) rectangle (6,6);
\draw [pattern=dots] (8,4) rectangle (10,5);
\draw [pattern=dots] (12,3) rectangle (14,4);

\draw [pattern=dots] (18,2) rectangle (19,3);
\draw [pattern=dots] (21,1) rectangle (23,2);
\draw [pattern=dots] (25,0) rectangle (27,1);

\draw [pattern=north east lines] (18,-1) rectangle (22,0);
\draw [pattern=dots] (31,-1) rectangle (35,0);

\draw [blue,decorate,decoration={zigzag,segment length=2mm,amplitude=.3mm}] (-0.1,-0.1) |- (3,7) |- (5,6) |- (8,5) |- (12,4) |- (18,3) |- (21,2) |- (25,1) |- (31,0) |- (39,-1) -- (39,-2);
\draw [blue,loosely dotted,thick] (0,0) -- (0,-2);
\draw [blue,loosely dotted,thick] (39,-2) -- (39,-3);

\draw [red,thick,->,decorate,decoration={snake,amplitude=.4mm}] (7,2) to (7,5);
\draw [<-,red] (7,3.5) to [out=10, in=260] (10,6) node[above] {\footnotesize{$k+1-t$}};
\draw [red,thick,->,decorate,decoration={snake,amplitude=.4mm}] (8,1.5) to node[above]{\footnotesize{$t$}} (21,1.5);

\draw [loosely dotted,thick] (0,0) -- (12,0);
\draw [loosely dotted,thick] (0,3) -- (3,3);
\draw [loosely dotted,thick] (6,3) -- (12,3);

\node [left] at (0,-0.5) {\footnotesize{$r$}};
\node [left] at (0,0.5) {\footnotesize{$r+1$}};
\node [left] at (0,2.5) {\footnotesize{$r+k+1-t$}};

\end{tikzpicture}
	\]	
	\[
	\left(
\text{Here }
\begin{cases}
\begin{tikzpicture}[scale=0.33]
	\draw (0,0)--(2,0);
\end{tikzpicture}
 & \text{: outline of $\core(\la)$} \\
\begin{tikzpicture}[scale=0.33]
	\draw[blue,decorate,decoration={zigzag,segment length=2mm,amplitude=.3mm}] (0,0)--(2,0);
\end{tikzpicture}
 & \text{: outline of $\core(\ti\la)$} \\
\begin{tikzpicture}[scale=0.33]
	\draw [pattern=north east lines] (0,0) rectangle (2,1);
\end{tikzpicture}
& \text{: $\core(\nu)/\core(\la)$} \\
\begin{tikzpicture}[scale=0.33]
	\draw [pattern=dots] (0,0) rectangle (2,1);
\end{tikzpicture}
& \text{: $\core(\ti\nu)/\core(\ti\la)$}
\end{cases}
\right)
	\]
	Then,
	\begin{enumerate}
		\item if $i < r+(k+1-t)$, \\ $(i,\core(\ti\la)_i)$ is a $\core(\ti\la)$-removable corner $\iff$ $(i,\core(\la)_i)$ is a $\core(\la)$-removable corner,
		\item if $i \ge r+(k+1-t)$, \\
			$(i,\core(\ti\la)_i)$ is a $\core(\ti\la)$-removable corner $\iff$ $(i-(k+1-t),\core(\la)_{i-(k+1-t)})$ is a $\core(\la)$-removable corner.
	\end{enumerate}
	Moreover,
	when $(i,\core(\ti\la)_i)$ is a $\core(\ti\la)$-removable corner (of residue $a$),
	we consider two cases:
	\begin{enumerate}
		\item if $i < r+(k+1-t)$. Then
	\begin{align*}
		\text{$(i,\core(\ti\la)_i)$ is $\core(\ti\nu)$-blocked} 
		&\iff \core(\ti\la)_i \le \core(\ti\nu)_{i+1} \\
		&\iff \core(\la)_i+t \le \core(\nu)_{i+1}+t \\
		&\iff \text{$(i,\core(\la)_i)$ is $\core(\nu)$-blocked},
	\end{align*}
	and the residue of $(i,\core(\la)_i)$ is $a-t$.
		\item if $i \ge r+(k+1-t)$. Then
	\begin{align*}
		\text{$(i,\core(\ti\la)_i)$ is $\core(\ti\nu)$-blocked} 
		&\iff \core(\ti\la)_i \le \core(\ti\nu)_{i+1} \\
		&\iff \core(\la)_{i-(k+1-t)} \le \core(\nu)_{i+1-(k+1-t)} \\
		&\iff \text{$(i-(k+1-t),\core(\la)_{i-(k+1-t)})$ is $\core(\nu)$-blocked},
	\end{align*}
	and the residue of $(i-(k+1-t),\core(\la)_{i-(k+1-t)})$ is $a-t$.
	\end{enumerate}

	Hence, for each $a \in \Z/(k+1)$, 
	there exists a non-$\core(\ti\nu)$-blocked $\core(\ti\la)$-removable $a$-corner
	if and only if
	there exists a non-$\core(\nu)$-blocked $\core(\la)$-removable $(a-t)$-corner.
	Therefore we have $r_{\core(\nu)\core(\la)} = r_{\core(\ti\nu)\core(\ti\la)}$.
\end{proof}

As a corollary of the proof of the above lemma, we have
\begin{coro}\label{theo:r_Rt}
	For any $\la,\nu\in\Pk$ and $1\le t\le k$ we have
$r_{\core(R_t\cup\nu)\core(R_t\cup\la)}=r_{\core(\nu)\core(\la)}$ or $r_{\core(\nu)\core(\la)}+1$.
\end{coro}
\begin{proof}
	Take $r$ such that $\la_r\ge t>\la_{r+1}$ and do a same argument as the above lemma.

	Then we have that, if $i\neq r+(k+1-t)$, 
	there exists a $\core(\ti\nu)$-nonblocked $\core(\ti\la)$-removable $a$-corner in $i$-th row
	if and only if
	there exists a $\core(\nu)$-nonblocked $\core(\la)$-removable $(a-t)$-corner in $i'$-th row.
	(Here we put $i'=i$ if $i<r+(k+1-t)$ and $i'=i-(k+1-t)$ if $i>r+(k+1-t)$)

	Hence we have $r_{\core(\nu)\core(\la)} \le r_{\core(\ti\nu)\core(\ti\la)} \le r_{\core(\nu)\core(\la)} + 1$.
\end{proof}

\begin{prop}\label{Rt_twice}
For $\lambda \in \Pk$ and $1 \le t \le k$, we have
$\kks{\lambda\cup R_t \cup R_t} = \kks{\lambda \cup R_t} \cdot \displaystyle\frac{\kks{R_t\cup R_t}}{\kks{R_t}}$.
\end{prop}

\begin{proof}
Write $\tilde\mu = \mu \cup R_t$ for $\mu \in \Pk$.

Define a linear map $\Theta : \Lk \lra \Lk$ by $\kks{\mu} \lmt \kks{\tilde\mu}$ for all $\mu\in\Pk$
and put $X = \mathrm{span}\{\kks{\tilde\la}\mid \lambda\in\Pk\}$.
Then $X$ is an ideal of $\Lk$ 
because $h_r\cdot \kks{\tilde\la}$ can be written as a linear combination of $\{\kks{\tilde\nu}\mid \nu\in \Pk\}$,
by (\ref{eq:Pieri}) and Proposition \ref{weaks_rect_P}.

Next we claim
$$ \Theta|_X \circ (h_r\cdot) = (h_r\cdot) \circ \Theta|_X\ \ :X\lra X$$
for $1\le r \le k$,
where $h_r\cdot$ denotes the multiplication by $h_r$.

\vspace{2mm}
{\it Proof of claim.}

It suffices to show $h_r\cdot \kks{\widetilde{\mu\cup R_t}} = \Theta(h_r\cdot \kks{\tilde\mu})$ for $\mu \in \Pk$.
More generally, we can show $h_r\cdot \kks{\widetilde{\mu\cup (t)}} = \Theta(h_r\cdot \kks{\mu \cup (t)})$ for $\mu \in \Pk$:

\begin{align*}
	h_r\cdot \kks{\widetilde{\mu\cup (t)}} 
	&= \sum_{s=0}^r (-1)^{r-s} 
	   \sum_{\substack{\eta \\ \core(\eta)/\core(\widetilde{\mu\cup(t)})\text{ is a weak $s$-strip}}} \binom{r_{\core(\eta),\core(\widetilde{\mu\cup(t)})}}{r-s} \kks{\eta} \\
	&= \sum_{s=0}^r (-1)^{r-s} 
	   \sum_{\substack{\nu \\ \core(\nu)/\core(\mu\cup(t))\text{ is a weak $s$-strip}}} \binom{r_{\core(\tilde\nu),\core(\widetilde{\mu\cup(t)})}}{r-s} \kks{\tilde\nu} \\
	&= \sum_{s=0}^r (-1)^{r-s}
	   \sum_{\substack{\nu \\ \core(\nu)/\core(\mu\cup(t))\text{ is a weak $s$-strip}}} \binom{r_{\core(\nu),\core(\mu\cup(t))}}{r-s} \kks{\tilde\nu} \\
	&= \Theta\bigg(\sum_{s=0}^r (-1)^{r-s} \sum_{\substack{\nu \\ \core(\nu)/\core(\mu\cup(t))\text{ is a weak $s$-strip}}} \binom{r_{\core(\nu),\core(\mu\cup(t))}}{r-s} \kks{\nu}\bigg) \\
	&= \Theta\bigg(h_r\cdot \kks{\mu\cup(t)}\bigg).
\end{align*}
Here the second equality uses Proposition \ref{weaks_rect_P}, and the third equality uses Lemma \ref{p=p}.
Hence the claim is proved.

\vspace{2mm}
Since $h_1,\ldots,h_k$ generate $\Lk$, the claim implies that $\Theta|_X$ is a $\Lk$-module homomorphism.
Hence for any $x \in X$, 
\[ x \cdot \Theta(\kks{R_t}) = \Theta(x\kks{R_t}) = \Theta(x) \cdot \kks{R_t}, \]
which implies
$\Theta(x) = x \cdot \displaystyle\frac{\kks{R_t\cup R_t}}{\kks{R_t}}$ for any $x \in X$.
Setting $x=\kks{R_t\cup\la}$ gives the proposition.
\end{proof}


\begin{theo}\label{Result_two}
Let $P=\RRma$ be as in \NP, and
put $Q=R_{t_1}\cup\cdots\cup R_{t_m}$.
Then, for $\la\in\Pk$ we have
$$ \frac{\kks{P\cup\la}}{\kks{P}}=\frac{\kks{Q\cup\la}}{\kks{Q}}.$$
\end{theo}
\begin{proof}
	Induction on $\sum_i (a_i-1)$.
	If $\sum_i (a_i-1)=0$ then it is obvious since $P=Q$.
	
	Otherwise, we can assume $a_1>1$ without loss of generality.
	Write $P=R_{t_1}\cup R_{t_1} \cup P'$.
	By Proposition \ref{Rt_twice} we have
	\[
		\frac{\kks{P'\cup\la\cup R_{t_1}\cup R_{t_1}}}{\kks{P'\cup\la\cup R_{t_1}}}
		= \frac{\kks{R_{t_1}\cup R_{t_1}}}{\kks{R_{t_1}}}
		= \frac{\kks{P'\cup R_{t_1}\cup R_{t_1}}}{\kks{P'\cup R_{t_1}}},
	\]
	thus we conclude
	\[
		\frac{\kks{P'\cup\la\cup R_{t_1}\cup R_{t_1}}}{\kks{P'\cup R_{t_1}\cup R_{t_1}}}
		= \frac{\kks{P'\cup\la\cup R_{t_1}}}{\kks{P'\cup R_{t_1}}}
		= \frac{\kks{Q\cup\la}}{\kks{Q}}.
	\]
	Here we used induction hypothesis for the second equality.
\end{proof}

\section{A factorization of $\kks{\RRma\cup(r)}$}\label{sec:one_row}

In this section we will give an explicit formula for $\kks{\RRma\cup\la}\big/\kks{\RRma}$ when $\la=(r)$.

Roughly speaking, $K$-$k$-Schur functions can be calculated by ``solving'' the system of Pieri rule formulas (\ref{eq:Pieri}).
To solve such a system, it is important to understand concretely what weak strips $\core(\nu)/\core(\mu)$ are.

If $\mu$ is
a union of $k$-rectangles $P=\RRma$ the situation is simple: 
if $\core(\nu)/\core(P)$ is a weak strip then $\nu$ has the form $P\cup(s)$ for some $s$,
as we will see in the proof of the following proposition.
Thus the Pieri rule also has a simple explicit expression as follows:

\begin{prop}\label{Rt_prod_hr}
Let $P$ and $\al{u}$ $(u \in \Z_{>0})$ be as in \NP in Section \ref{sec:general}, before Proposition \ref{P_factor}. 
Then, for $1 \le r \le k$, we have
$$\kks{P}\cdot h_r = \sum_{s=0}^{r}(-1)^{r-s}\binom{\al{s+1}}{r-s}\kks{P\cup(s)}.$$
\end{prop}
\begin{proof}
We have
$\core(P) = \underbrace{R_{t_m} \oplus \dots \oplus R_{t_m}}_{a_m} \oplus \cdots \oplus \underbrace{R_{t_1} \oplus \dots \oplus R_{t_1}}_{a_1}$
and
all addable corners of $\core(P)$ has the same residue, say $i$.
Moreover, $\core(P)$ has a total of $\sum_j a_j$ removable corners,
$a_j$ of which are derived from
the removable corner of $R_{t_j}$ and
having the residue $i+t_j$ for each $j$.

Next we claim that if $\gamma/\core(P)$ is a weak $s$-strip
then $\gamma = s_{i+s-1}\cdots s_{i+1}s_{i}(\core(P))$.

\vspace{2mm}
{\it Proof of the claim.}
We prove it by induction on $s$.
If $s=1$, it is obvious because all addable corners of $\core(P)$ have
the same residue $i$.

Let $s>1$ and $\gamma/\core(P)$ be a weak $s$-strip.
Then we can write $\gamma = s_{j_s}\cdots s_{j_2}s_{j_1}(\core(P))$, 
where $(j_s,\cdots,j_1)$ is cyclically decreasing (see Definition-Proposition \ref{weakstrip}(4)).

Since $s_{j_{s-1}}\cdots s_{j_2}s_{j_1}(\core(P))/\core(P)$ is a weak
$(s-1)$-strip, we have $(j_{s-1},\cdots,j_1)=(i+s-2,\cdots,i+1,i)$ by the induction hypothesis.
Since $(j_s,i+s-2,\cdots,i+1,i)$ is cyclycally decreasing,
we have $j_s \notin \{i-1,i,i+1,\cdots,i+s-2\}$.

If $j_s \neq i+s-1$, then $s_{j_s}$ commutes with $s_i,s_{i+1},\cdots,s_{i+s-2}$ and
\begin{align*}
\gamma &= s_{j_s}s_{i+s-2}\cdots s_{i+1}s_{i}(\core(P)) \\
       &= s_{i+s-2}\cdots s_{i+1}s_{i}s_{j_s}(\core(P)). 
\end{align*}
However, $|s_{j_s}(\core(P))|_k \le |\core(P)|_k$ because $\core(P)$ doesn't have an addable corner of residue $j_s$.
Hence $|\gamma|_k \le |\core(P)|_k+s-1$, violating the assumption that $\gamma/\core(P)$ is a weak $s$-strip. 

Hence we have $j_s = i+s-1$, completing the proof of the claim.

\vspace{2mm}
Since $s_{i+s-1}\cdots s_{i+1}s_{i}(\core(P))$ has the form below on the right,
we can see that the corresponding $k$-bounded partition has the form $P\cup(s)$.
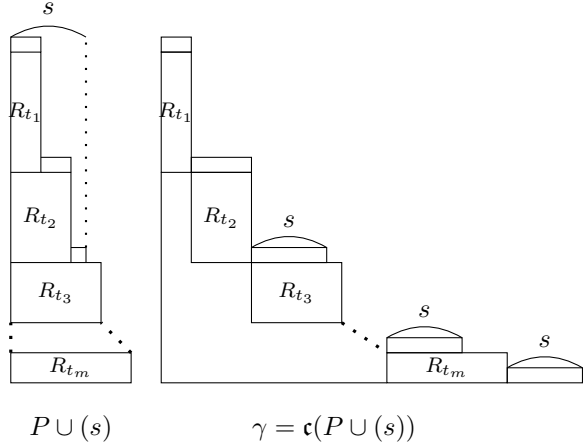
\begin{figure}[htbp]
\[
	\begin{tikzpicture}[scale=0.20]
\draw (0,0) -- (8,0) -- (8,2) -- (0,2) --cycle;

\draw (0,4) -- (6,4) -- (6,8) -- (4,8) -- (4,14) -- (2,14) -- (2,22) -- (0,22) -- cycle;
\draw (0,8) -- (4,8);
\draw (4,9) -- (5,9) -- (5,8);
\draw (2,15) -- (4,15) -- (4,14);
\draw (0,22) -- (0,23) -- (2,23) -- (2,22);
\draw (0,14) -- (2,14);
\draw [loosely dotted,very thick] (0,2) -- (0,4);
\draw [loosely dotted,very thick] (8,2) -- (6,4);
\draw [loosely dotted,thick] (5,9) -- (5,23);

\node at (1,18) {\footnotesize $R_{t_1}$};
\node at (2,11) {\footnotesize $R_{t_2}$};
\node at (3,6) {\footnotesize $R_{t_3}$};
\node at (4,1) {\footnotesize $R_{t_m}$};

\node[below=0.3] at (4,0) {$P\cup (s)$};

\draw (0,23) to [out=30,in=150] node[above=1pt]{$s$} (5,23);

\draw (10,14) rectangle (12,22);
\draw (10,22) rectangle (12,23);
\node at (11,18) {\footnotesize $R_{t_1}$};
\draw (12,8) rectangle (16,14);
\draw (12,14) rectangle (16,15);
\node at (14,11) {\footnotesize $R_{t_2}$};
\draw (16,4) rectangle (22,8);
\draw (16,8) rectangle (21,9);
\node at (19,6) {\footnotesize $R_{t_3}$};
\draw (25,0) rectangle (33,2);
\node at (29,1) {\footnotesize $R_{t_m}$};

\draw [loosely dotted,very thick] (22,4) -- (25,2);
\node [below=0.3] at (21.5,0) {$\ga=\core(P\cup(s))$};

\draw (16,9) to [out=30,in=150] node[above]{$s$} (21,9);
\draw (25,2) rectangle (30,3);
\draw (25,3) to [out=30,in=150] node[above]{$s$} (30,3);
\draw (33,0) rectangle (38,1);
\draw (33,1) to [out=30,in=150] node[above]{$s$} (38,1);

\draw (10,14) |- (25,0);

\end{tikzpicture}
\]
\caption{the shapes of $P\cup(s)$ and $\core(P\cup(s))$. In this figure $a_i=1$ for all $i$.}
\label{fig:c_P_s}
\end{figure}

Now we get back to the proof of the proposition.
Let $\gamma = s_{i+s-1}\cdots s_{i+1}s_{i}(\core(P))$.
Then the removable corner of $\core(P)$ corresponding to the removable corner of $R_{t_a}$ is $\gamma$-blocked if and only if $s\ge t_a$.
Then the number of residues of $\gamma$-nonblocked removable corners of $\core(P)$ is
exactly $\al{s+1}$.
\end{proof}

The above proposition gives an expression for $\kks{P} h_r$ as a linear combination of $\{\kks{P\cup(s)}\}_s$.
To solve this linear equation, we need the following lemma of binomial coefficients.

\begin{lemm}\label{binom_inv_matrix}
Let $l$ be a positive integer and $\beta_1,\beta_2,\cdots,\beta_{l+1}$ be integers such that
$\beta_i\ge \beta_{i+1} \ge \beta_{i}-1$ for each $i$.

Let $C=\bigg((-1)^{r-s}\displaystyle\binom{\beta_{s+1}}{r-s}\bigg)_{r,s=0}^{l}$.
Then $C^{-1} = \displaystyle\bigg(\binom{\beta_r + r-s-1}{r-s}\bigg)_{r,s=0}^{l}$.

Here $\displaystyle\binom{a}{b}$ is considered to be $0$ if $b<0$.
\end{lemm}
\begin{proof}
Put $D=\displaystyle\bigg(\binom{\beta_r + r-s-1}{r-s}\bigg)_{r,s=0}^{l}$.
The $(p,q)$ element of the matrix $DC$ is
\begin{align*}
(DC)_{pq} &= \sum_{i=0}^{l} \binom{\beta_p + p-i-1}{p-i} \cdot (-1)^{i-q}\displaystyle\binom{\beta_{q+1}}{i-q}\\
          &= \sum_{i=q}^{p} \binom{\beta_p + p-i-1}{p-i} \cdot (-1)^{i-q}\displaystyle\binom{\beta_{q+1}}{i-q}\\
          &= \sum_{j=0}^{p-q} \binom{\beta_p + p-q-j-1}{p-q-j} \cdot (-1)^{j}\displaystyle\binom{\beta_{q+1}}{j},
\end{align*}
which is 0 unless $p\ge q$.

Let us consider the case $p\ge q$.

By applying the next lemma for $a=\beta_{q+1},\ b=\beta_{p},\ c=p-q\ge 0$, we have
\begin{align*}
(DC)_{pq} &= (-1)^{p-q} \binom{\beta_{q+1}-\beta_{p}}{p-q} \\
          &= \begin{cases}
	  0 & \ ({\rm if}\ p>q), \\
	  1 & \ ({\rm if}\ p=q),
	  \end{cases}
\end{align*}
where the last equality follows from $\beta_{q+1}-\beta_{p} \in \{0,1,\cdots,p-q-1\}$ (if $q+1\le p$).
\end{proof}

\begin{lemm}\label{binom_lem1}
For integers $a$,$b$ and a nonnegative integer $c$,
\[
	\sum_{i=0}^{c} (-1)^i \binom{a}{i} \binom{b-1+c-i}{c-i} = (-1)^c \binom{a-b}{c}.
\]
\end{lemm}
\begin{proof}
	Since $\binom{m}{n}$ is the coefficient of $X^n$ in $(1+X)^m\in\Z[[X]]$ for $m\in\Z$ and $n\in\Z_{\ge 0}$, we have
	\begin{align*}
		(\textrm{LHS}) &= \sum_{i=0}^{c} (-1)^c \binom{a}{i} \binom{-b}{c-i} \\
			  &= (-1)^c \text{(the coefficient of $X^c$ in $(1+X)^a(1+X)^{-b}\in\Z[[X]]$)}\\
			  &= (-1)^c \binom{a-b}{c}.
	\end{align*}
\end{proof}

Now we can deduce the formula showing the goal of this section.

\begin{theo}\label{R_t_r}
If $P, \al{u}$ and $r$ are as in
Proposition \ref{Rt_prod_hr}, then we have
\[\frac{\kks{P\cup(r)}}{\kks{P}} = \sum_{s=0}^{r}\binom{\al{r}+r-s-1}{r-s}h_s.\]

In particular, if $t_m < r$, which means $\al{r}=0$, we have
\[\frac{\kks{P\cup(r)}}{\kks{P}} = h_r = \kks{(r)}\]

On the other hand, when $m=1$,
\[
\frac{\kks{R_t\cup(r)}}{\kks{R_t}} = 
\begin{cases}
h_r & ({\rm if}\ r>t), \\
h_r+h_{r-1}+\cdots+h_0 & ({\rm if}\ r\le t).
\end{cases}
\]
\end{theo}

\begin{proof}
	Apply Lemma \ref{binom_inv_matrix} for Proposition \ref{Rt_prod_hr}.
\end{proof}

\section{A factorization of $g^{(k)}_{\RRma\cup\la}$ with small $\la$ and splitting $\kks{\RRma}$ into $\kks{R_{t_1}^{a_1}}\dots\kks{R_{t_m}^{a_m}}$}\label{sec:dist}
\subsection{Statements}\label{sec:dist_ex}
Our goal in this section is to show the equality
$$\kks{R_{t_1}^{a_1}\cup\cdots\cup R_{t_m}^{a_m}} = \kks{R_{t_1}^{a_1}}\cdots\kks{R_{t_m}^{a_m}}$$
for 
$1\le t_1 < \cdots < t_m \le k$ and $a_i>0$ (see Theorem \ref{dist_krec_factor}).

The essential part is to prove $\kks{R_{t_1}\cup\dots\cup R_{t_{m}}} = \kks{R_{t_1}\cup\dots\cup R_{t_{m-1}}} \kks{R_{t_m}}$,
and the remainging part follows from the results from Section \ref{sec:general} and induction (on $n$).
This is a simple statement, but our proof involves an induction on the shape of partitions, 
thus we have to prove a more general statement
(see the case $t_n < r$ of part (2) of Theorem \ref{P_mu_r}):
Let $P=\bigcup_{i=1}^{m} R_{t_i}^{a_i}$ be as in \NP, Section \ref{sec:general}, before Proposition \ref{P_factor},
and $\la$ as follows:
\begin{itemize}
	\item [\Nla]\label{nota:Nla}
		Let $(\emptyset\neq)\la\in\Pk$ with satisfying $\bla\subset\Rbl$,
		where we write
		$\bla=(\la_1,\la_2,\dots,\la_{l(\la)-1})$ and
		$\bl = l(\bla) = l(\la) - 1$.
		(Here we consider $R_{t}$ to be $\emptyset$ unless $1\le t\le k$)
\end{itemize}
		({\it Note}:
		when $l(\la)=1$, we have $\bl = 0$ and $\bla = \emptyset = \Rbl$ thus $\la$ satisfies \Nla.
		When $l(\la) > k+1$, we have $\bl > k$ and $\bla \neq \emptyset = \Rbl$ thus $\la$ does not satisfy \Nla.
		)
		
Then, 
\begin{equation}\label{eq:dist_goal_small}
\kks{P\cup\la} = \kks{P}\kks{\la} 
\quad\text{when } \la_{l(\la)}>\max_i\{t_i\}.
\end{equation}

\subsection{Proofs}

We will prove a slightly even more general formula than (\ref{eq:dist_goal_small}) (see part (2) of Theorem \ref{P_mu_r}) in the following procedure.
\begin{itemize}
	\item
		\underline{Step (A)}:\\
		First we write $\kks{\la}$ as a linear combination of products of $h_i$'s and $\kks{\mu}$'s with $l(\mu)<l(\la)$:
		putting $\bar\la=(\la_1,\dots,\la_{l(\la)-1})$, we have
		\[
			\kks{\la} = \hspace{-7mm}\sum_{\substack{\mu \text{ s.t.}\\ \bar\la\subset\mu\subset R_{k-l(\bar\la)+1} \\ \mu/\bar\la\text{ :vertical strip}}}\hspace{-7mm} (-1)^{|\mu/\bar\la|} \kks{\mu} \sum_{i\ge 0} 
			\binom{(|\mu/\bar\la|+r_{\mu'\bar\la'})+i-1}{i} h_{\la_{l(\la)}-|\mu/\bar\la|-i}
		\]
		if $\bar\la_1+l(\bar\la)\le k+1$
		(Lemmas \ref{kks_prod_hr} (1), \ref{binom_heikouidou} (1), \ref{g_mu_r_expand} (1) and \ref{A_la_q}).
	\item
		\underline{Step (B)}:\\
		Derive a similar expression for $\kks{P\cup\la}$
		(parts Lemmas \ref{kks_prod_hr} (2)-(3), \ref{binom_heikouidou} (2)-(3), \ref{g_mu_r_expand} (2)-(3), \ref{A_la_q}).
	\item
		\underline{Step (C)}:\\
		Compare (B) with the equality obtained by multiplying the formula in Step (A) by $\kks{P}$,
		noticing $\kks{P}\kks{\mu}=\kks{P\cup\mu}$
		by induction.
\end{itemize}

\vspace{2mm}
Step (A) consists of two substeps:
\begin{itemize}
	\item
		Step (A-1):
		Write down the Pieri rule for $\kks{\mu} h_r$ explicitly.
	\item
		Step (A-2):
		Solve the system of Pieri rule formulas to give expressions for $\kks{\la}$ as a linear combination of $\{\kks{\mu} h_r\}_{\mu,r}$.
\end{itemize}

Obtaining an expression for $\kks{P\cup\la}$ in Step (B) follows from similar steps (B-1) and (B-2).
\begin{itemize}
	\item
		Step (B-1):
		Write down the Pieri rule for $\kks{P\cup\mu} h_r$ explicitly.
	\item
		Step (B-2):
		Solve the system of Pieri rule formulas to give expressions for $\kks{P\cup\la}$ as a linear combination of $\{\kks{P\cup\mu} h_r\}_{\mu,r}$.
\end{itemize}

\subsection{Steps (A-1) and (B-1)}
Toward Step (A-1) and (B-1),
let us begin with describing weak strips $\core(\la)/\core(\mu)$ where $\mu$ is contained in a $k$-rectangle.

\begin{lemm}\label{ws_in_R}
	Assume $\mu \subset \Rl$ and $\mu_l>0$.
	Let $0 \le u \le \mu_l$ be an integer.

	\noindent{$(1)$}
	For $\ka\in\Pk$,
	\begin{align*}
		\text{$\core(\ka)/\core(\mu)$ is a weak $u$-strip} &\Longleftrightarrow
	\begin{cases}
	\text{$\ka/\mu$ is a horizontal $u$-strip}, \\ 
	\ka_1\le k-l+1,
	\end{cases} \\
	&\Longleftrightarrow \ka=\nu\cup(s) , \\
	&\phantom{\Longleftrightarrow} \text{where}
	\begin{cases}
	\nu \subset \Rl, \\
	\text{$\nu/\mu$ is a horizontal strip of size $\le u$}, \\
	s = u-|\nu/\mu|.
	\end{cases}
	\end{align*}
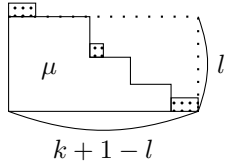
	\[
\begin{tikzpicture}[scale=0.18]
\draw (0,0)--(12,0)--(12,2)--(9,2)--(9,4)--(6,4)--(6,7)--(0,7)--cycle;
\draw[loosely dotted,thick] (0,7)-|(14,0);
\draw[pattern=dots] (0,7) rectangle (2,8);
\draw[pattern=dots] (6,4) rectangle (7,5);
\draw[pattern=dots] (12,0) rectangle (14,1);
\node at (3,3) {$\mu$};
\draw (0,0) to [out=-20,in=-160] node[below]{$k+1-l$} (14,0);
\draw (14,0) to [out=70,in=-70] node[right]{$l$} (14,7);
\end{tikzpicture}
\]

	\noindent{$(2)$}
	For $\ti\ka\in\Pk$, 
	\begin{align*}
	&\text{$\core(\tilde\ka)/\core(P\cup\mu)$ is a weak $u$-strip} \\
	 &\quad\iff \tilde\ka=P\cup\ka, \text{where $\core(\ka)/\core(\mu)$ is a weak $u$-strip} \\
	&\quad\iff \tilde\ka=P\cup\nu\cup(x), \text{where} 
		\begin{cases} 
			\nu\subset\Rl, \\ 
			\text{$\nu/\mu$: horizontal strip}, \\ 
			|\nu/\mu|+x = u.
		\end{cases}
	\end{align*}

\end{lemm}

\begin{proof}
	\noindent{(1):}
	The second equivalence is obvious.
	
	The ``if'' part of the first equivalence is easy:
	since $\ka_1\le k+1-l$ and $l(\ka)\le l+1$, 
	we have $\core(\ka)=\ka$ or $\core(\ka)=(\ka_1+\ka_{l+1},\ka_2,\ldots)$ and hence $\kconj{\ka}/\kconj{\mu}$ is a vertical strip.

	Hence it suffices to prove the ``only if'' part of the first equivalence:
	let
	$\mu\subset R_{k+1-l}$, $\mu_l>0$, and $\core(\ka)/\core(\mu)$ be a weak strip of size $\le \mu_l$, 
	and we shall prove $\ka_1 \le k+1-l$.

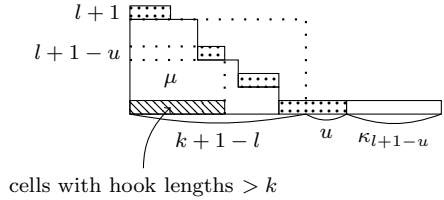
\[
\begin{tikzpicture}[scale=0.18]
\draw (0,0)--(11,0)--(11,2)--(8,2)--(8,4)--(5,4)--(5,7)--(0,7)--cycle;
\draw[loosely dotted,thick] (0,7)-|(13,0);
\draw[pattern=dots] (0,7) rectangle (3,8);
\draw[pattern=dots] (5,4) rectangle (7,5);
\draw[pattern=dots] (8,2) rectangle (11,3);
\draw[pattern=dots] (11,0) rectangle (16,1);
\node at (3,2.5) {\footnotesize$\mu$};
\draw (0,0) to [out=-10,in=-170] node[below]{\footnotesize$k+1-l$} (13,0);
\draw (13,0) to [out=-30,in=-150] node[below]{\footnotesize$u$} (16,0);
\draw (16,0) to [out=-20,in=-160] node[below]{\footnotesize$\ka_{l+1-u}$} (23,0);

\draw[loosely dotted,thick] (0,4)--(5,4);
\draw[loosely dotted,thick] (0,5)--(5,5);
\draw[loosely dotted,thick] (7,0)--(7,4);
\draw (16,0) rectangle (23,1);
\draw [pattern=north west lines] (0,0) rectangle (7,1);
\node [left] at (0,4.5) {\footnotesize$l+1-u$};
\node [left] at (0,7.5) {\footnotesize$l+1$};

\draw [<-] (3,0.5) to [out=-135,in=90] (1,-4) 
	node[below]{\footnotesize cells with hook lengths $>k$};

\end{tikzpicture}
\]

	Assume, on the contrary, that $\ka_1 > k+1-l$.
	Write $\bar\ka=(\ka_2,\ka_3,\dots)\subset \Rl$.
	Then by Lemma \ref{core_j} we have
	\[
		\core(\ka)_i = \core(\bar{\ka})_{i-1} = \bar{\ka}_{i-1}=\ka_i\ \ \ \text{(for $i>1$), and}
	\]
	\[
		\core(\ka)_1 = \ka_1 + \underbrace{\core(\ka)_{\scriptsize{\underbrace{1+k+1-\ka_1}_{< l+1}}}}_{\ge \ka_l \ge \mu_l} \ge \ka_1+\mu_l > k+1-l+\mu_l.
	\]

	Hence, the hook lengths of $(1,1),\cdots, (1,\mu_l)$ in $\core(\ka)$ are all greater than $k$ because 
	\begin{align*}
	h_{(1,j)}(\core(\ka)) &= \core(\ka)_1 + \core(\ka)'_j-1-j+1 \\
	&> k+1-l+\mu_l + \underbrace{\core(\ka)'_j}_{\ge \ka'_j\ge \mu'_j \ge l} \underbrace{- j}_{\ge -\mu_l} \\
	       &\ge k+1
	\end{align*}
	for $1\le j \le \mu_l$.
	On the other hand, those of $(2,1),\cdots, (2,\mu_l)$ in $\core(\ka)$ are less than or equal to $k$ because 
	$\bar\ka \subset \mu \subset \Rl$.

	Hence 
	\[
		\ka^{\omega_k}_j = \core(\ka)'_j - 1 = \ka'_j - 1
	\]
	for $1\le j \le \mu_l$.

	Since $\core(\ka)/\core(\mu)$ is a weak strip, $\kconj{\ka}/\kconj{\mu}$ is a vertical strip.
	Hence 
	\[ \ka'_j - 1 = \kconj{\ka}_j \ge \kconj{\mu}_j = \mu'_j = l\]
	for $1 \le j \le \mu_l$, which implies $\ka_{l+1} \ge \mu_l$.
	Then we have 
	\[ |\ka/\mu| \ge (\ka_{l+1}-\underbrace{\mu_{l+1}}_{=0}) + (\underbrace{\ka_1-\mu_1}_{>0}) > \mu_l \]
	since $\ka_1>k+1-l\ge \mu_1$.
	This is a contradiction.
	
	\noindent{(2):}
	The first equivalence follows from Corollary \ref{weaks_rect_P}.
	The second equivalence follows from (1).
\end{proof}

Next let us explicitly describe the weak Pieri rule (\ref{eq:Pieri}),
after we prepare a notation for convenience.

\begin{defi}\label{def:T}
	Let $P$ and $\al{u}$ $(u\in\Z_{>0})$ be as in \NP.
	For $\nu \subset R_{k+1-l(\nu)}$, $0\le u \le \nu_{l(\nu)}$, $p \in \mathbb{Z}$,
	we set
	\begin{align*}
		T_{\nu,u,p} &:= \sum_{s=0}^{u} (-1)^s \binom{p}{s} \kks{\nu\cup (u-s)}, \\
		T'_{P,\nu,u,p} &:= \sum_{s=0}^{u} (-1)^s \binom{p+\al{u+1-s}}{s} \kks{P\cup\nu\cup (u-s)}.
	\end{align*}
\end{defi}
\begin{lemm}\label{kks_prod_hr} 
	Let $P$ and $\al{u}$ $(u\in\Z_{>0})$ be as in \NP.
	Assume $\mu \subset \Rl$, $\mu_l>0$, $\mu_l\ge r \ge 0$.
	Then we have

$(1)$
\vspace{-3mm}
\begin{align*}
\kks{\mu} h_r &= \sum_{\substack{\nu\subset R_{k+1-l} \\ \nu/\mu\text{\rm:h.s.}}} 
		\sum_{s=0}^{r-|\nu/\mu|} (-1)^s \binom{r_{\nu\mu}}{s} \kks{\nu\cup (r-|\nu/\mu|-s)} \\
		\Bigg(&= \sum_{\substack{\nu\subset R_{k+1-l} \\ \nu/\mu\text{\rm:h.s.}}} T_{\nu,r-|\nu/\mu|,r_{\nu,\mu}}\Bigg).
\end{align*}

$(2)$
If $\maxti<\mu_l$,
\begin{align*}
\kks{P\cup\mu} h_r &= \sum_{\substack{\nu\subset R_{k+1-l} \\ \nu/\mu\text{\rm:h.s.}}} 
	\sum_{s=0}^{r-|\nu/\mu|} (-1)^s \binom{r_{\nu\mu}+\al{r-|\nu/\mu|+1-s}}{s} \kks{P\cup\nu\cup (r-|\nu/\mu|-s)} \\
	\Bigg(&= \sum_{\substack{\nu\subset R_{k+1-l} \\ \nu/\mu\text{\rm:h.s.}}} T'_{P,\nu,r-|\nu/\mu|,r_{\nu,\mu}}\Bigg).
\end{align*}

$(3)$
If $\maxti=\mu_l$,
\begin{align*}
\kks{P\cup\mu} h_r &= \sum_{\substack{\nu\subset R_{k+1-l} \\ \nu/\mu\text{\rm:h.s.}}} 
\sum_{s=0}^{r-|\nu/\mu|} (-1)^s \binom{r_{\nu\mu}+\al{r-|\nu/\mu|+1-s}-1}{s} \kks{P\cup\nu\cup (r-|\nu/\mu|-s)} \\
		\Bigg(&= \sum_{\substack{\nu\subset R_{k+1-l} \\ \nu/\mu\text{\rm:h.s.}}} T'_{P,\nu,r-|\nu/\mu|,r_{\nu,\mu}-1}\Bigg).
\end{align*}
\end{lemm}

\begin{proof}
(1)
We transform the right-hand side of Eq.\ (\ref{eq:Pieri}),
Proposition \ref{Pieri}, into the right-hand side of part (1) of the Lemma as follows:
\begin{align}
\kks{\mu} h_r &= \sum_{u=0}^{r} (-1)^{r-u} 
	\sum_{\substack{\ka \\ \text{$\core(\ka)/\core(\mu)$:weak $u$-strip}}} 
	\binom{r_{\core(\ka)\core(\mu)}}{r-u} \kks{\ka} \notag\\
	&\underset{\text{(i)}}{=} \sum_{u=0}^{r} (-1)^{r-u} 
\sum_{\substack{\nu \text{ s.t.} \\ \nu\subset\Rl \\ \nu/\mu\text{: h.s.\ of size}\le u}} 
	   \binom{r_{\core(\nu\cup(u-|\nu/\mu|)),\core(\mu)}}{r-u} \kks{\nu\cup(u-|\nu/\mu|)} \notag\\
	   &= \sum_{\substack{\nu \text{ s.t.} \\ \nu\subset\Rl \\ \nu/\mu\text{: h.s.}}} 
	   \sum_{u=|\nu/\mu|}^{r} (-1)^{r-u} 
	   \binom{r_{\core(\nu\cup(u-|\nu/\mu|)),\core(\mu)}}{r-u} \kks{\nu\cup (u-|\nu/\mu|)} \label{eq:P_mu}\\
	   &\underset{\text{(ii)}}{=} \sum_{\substack{\nu \text{ s.t.} \\ \nu\subset\Rl \\ \nu/\mu\text{: h.s.}}} 
	   \sum_{s=0}^{r-|\nu/\mu|} (-1)^{s} 
	   \binom{r_{\nu,\mu}}{s} \kks{\nu\cup (r-s-|\nu/\mu|)}.\notag
	\end{align}
	
	Here, the equality (i) uses Lemma \ref{ws_in_R} (1)
	in order to change the summation variable from $\ka$ to $\nu$ according to 
	$\ka=\nu\cup(u-|\nu/\mu|)$.
	
	For the equality (ii)
	we use
	(1) of the following Lemma \ref{r_henkei} and put $s=r-u$.
	Note that
	$u-|\nu/\mu| \ge \mu_l$ occurs only if $u-|\nu/\mu| = u = r = \mu_l$ since $u \le r \le \mu_l$,
	in which case we have
	$\displaystyle\binom{r_{\nu\cup(u-|\nu/\mu|),\mu}}{r-u} = 1 = \displaystyle\binom{r_{\nu,\mu}}{r-u}$.
	
	\vspace{2mm}
	We can prove (2) and (3) almost the same as (1),
	using Lemma \ref{ws_in_R} (2) for (i),
	and (2) of the following Lemma \ref{r_henkei} for (ii).

	Note that, in the same way as (1), the case $u-|\nu/\mu|\ge\mu_l$ and $\maxti<\mu_l$
	appears in the expression for $\kks{P\cup\la}$ corresponding to (\ref{eq:P_mu})
	only in the form 
	$\DS\binom{r_{\nu\mu}+\al{\mu_l+1}-1}{0}$, 
	which is equal to $\DS\binom{r_{\nu\mu}+\al{\mu_l+1}}{0}$.

\end{proof}

\begin{lemm}\label{r_henkei}
	Let $\mu$ be as in Lemma \ref{kks_prod_hr}.
	Let $\mu\subset\nu\subset\Rl$ and assume $\nu/\mu$ is a horizontal strip.
	Let $0 \le x \le \nu_l$.
	Then we have
	
	\noindent{$(1)$}
	\vspace{-3mm}
	\[
		r_{\core(\nu\cup(x)),\core(\mu)} = r_{\nu\mu} - \de\left[x\ge\mu_l\right].
	\]

	\noindent{$(2)$}
	Let $P$ and $\al{u}$ $(u \in \Z_{>0})$ be as in Lemma \ref{kks_prod_hr} and assume $\maxti\le\mu_l$.
	Then
	\[
		r_{\core(P\cup\nu\cup(x)),\core(P\cup\mu)} = 
		\begin{cases}
			r_{\nu,\mu} + \al{x+1} - \de\left[x\ge\mu_l\right] & \text{$($if $\maxti<\mu_l)$},\\
			r_{\nu,\mu} + \al{x+1} - 1 & \text{$($if $\maxti=\mu_l$$)$}.
		\end{cases}
	\]

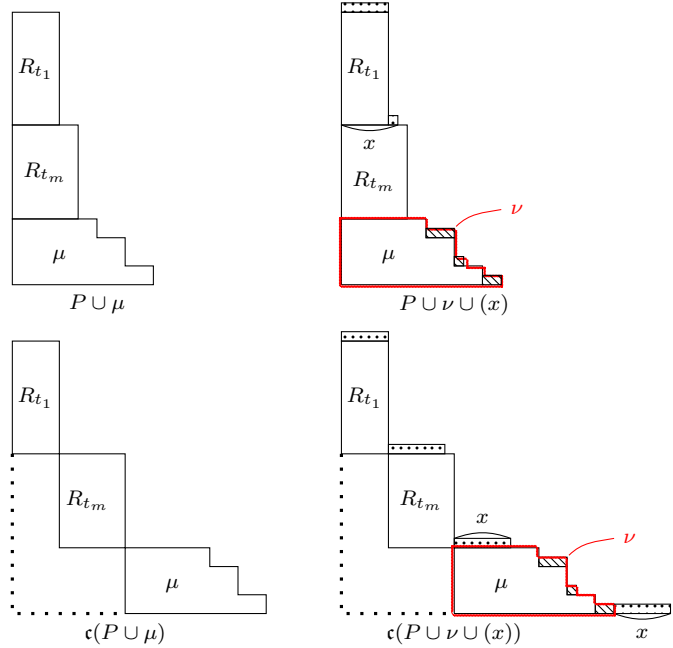
	\[
\begin{tikzpicture}[scale=0.125]
\draw (0,17) rectangle (5,29);
\node at (2.5,23) {\footnotesize $R_{t_1}$};
\draw (5,7) rectangle (12,17);
\node at (8,12) {\footnotesize $R_{t_m}$};
\draw (12,0)--(27,0)--(27,2)--(24,2)--(24,5)--(21,5)--(21,7)--(12,7)--cycle;
\node at (17,3) {\footnotesize $\mu$};

\draw[loosely dotted,very thick] (0,17) |- (12,0);
\node[below] at(12,0) {\footnotesize $\core(P\cup\mu)$};

\draw (0,52) rectangle (5,64);
\node at (2.5,58) {\footnotesize $R_{t_1}$};
\draw (0,42) rectangle (7,52);
\node at (3.5,47) {\footnotesize $R_{t_m}$};
\draw (0,35)--(15,35)--(15,37)--(12,37)--(12,40)--(9,40)--(9,42)--(0,42)--cycle;
\node at (5,38) {\footnotesize $\mu$};
\node[below] at (9,35) {\footnotesize $P\cup\mu$};

\draw (35,17) rectangle (40,29);
\node at (37.5,23) {\footnotesize $R_{t_1}$};
\draw (40,7) rectangle (47,17);
\node at (43.5,12) {\footnotesize $R_{t_m}$};
\draw (47,0) -| (62,2) -| (59,5) -| (56,7) -| (47,0);
\node at (52,3) {\footnotesize $\mu$};

\draw[loosely dotted,very thick] (35,17) |- (47,0);
\node[below] at (47,0) {\footnotesize $\core(P\cup\nu\cup(x))$};

\draw[pattern=dots] (35,29) rectangle (40,30);
\draw[pattern=dots] (40,17) rectangle (46,18);
\draw[pattern=dots] (47,7) rectangle (53,8);

\draw[pattern=north west lines] (56,5) rectangle (59,6);
\draw[pattern=north west lines] (59,2) rectangle (60,3);
\draw[pattern=north west lines] (62,0) rectangle (64,1);
\draw[pattern=dots] (64,0) rectangle (70,1);

\draw (47,8) to [out=20,in=160] node[above]{\footnotesize $x$} (53,8);
\draw (64,0) to [out=-20,in=-160] node[below]{\footnotesize $x$} (70,0);

\draw [red,decorate,decoration={snake,segment length=.2mm,amplitude=.1mm}] (46.9,-0.2) -- (64.1,-0.2) -- (64.1,1.0) -- (62.1,1.0) -- (62.1,1.95) --(60.1,1.95) -- (60.1,2.95) --(59.0,2.95) -- (59.0,5.95) -- (56.0,5.95) -- (56.0,6.95) -- (46.9,6.95) -- (46.9,-0.2);
\draw [red] (59.2,6.2) to [out=45,in=190] (64,8) node[right]{\footnotesize $\nu$};

\draw (35,52) rectangle (40,64);
\node at (37.5,58) {\footnotesize $R_{t_1}$};
\draw (35,42) rectangle (42,52);
\node at (38.5,46) {\footnotesize $R_{t_m}$};
\draw (35,35) -| (50,37) -| (47,40) -| (44,42) -| (35,35);
\draw [red,decorate,decoration={snake,segment length=.2mm,amplitude=.1mm}] (34.9,34.8) -| (52.1,36.10) -| (50.1,37.1) -| (48.1,38.1) -| (47.1,41.1) -| (44.15,42.10) -| (34.9,34.8);
\draw [red] (47.2,41.2) to [out=45,in=190] (52,43) node[right]{\footnotesize $\nu$};
\node at (40,38) {\footnotesize $\mu$};

\node[below] at (47,35) {\footnotesize $P\cup\nu\cup(x)$};

\draw[pattern=dots] (35,64) rectangle (40,65);
\draw[pattern=dots] (40,52) rectangle (41,53);

\draw[pattern=north west lines] (44,40) rectangle (47,41);
\draw[pattern=north west lines] (47,37) rectangle (48,38);
\draw[pattern=north west lines] (50,35) rectangle (52,36);

\draw (35,52) to [out=-20,in=-160] node[below] {\footnotesize $x$} (41,52);

\end{tikzpicture}
\]

\end{lemm}

\begin{proof}
	\noindent{$(1)$:}
	Since $\mu\subset \Rl$
	we have $\core(\mu) = \mu$.
	Since $\nu\subset \Rl$ and $x \le \nu_l$,
	we have $\core(\nu\cup(x))_i = (\nu\cup(x))_i$ for $i\neq 1$.
	Thus $r_{\core(\nu\cup(x)),\core(\mu)} = r_{\nu\cup(x),\mu}$. 
	Moreover,
	$r_{\nu\cup(x),\mu} \neq r_{\nu,\mu}$ happens only if
	the $(l+1)$-th part of $\nu\cup(x)$ blocks the $\mu$-removable corner in the $l$-th row, i.e.\ $x \ge \mu_l$, 
	in which case $r_{\nu\cup(x),\mu} = r_{\nu,\mu} - 1$.

	\noindent{$(2)$:}

	Assume $t_1 < \dots < t_m$ without loss of generality and thus $\maxti=t_m$.
	We put $T=\core(P)_1=\sum_j t_j$.
	Since $\core(P\cup\mu) = \mu\oplus\core(P)$,
	a removable corner $(r,c)$ of $\core(P\cup\mu)$ satisfies one of the following:
	\begin{itemize}
		\item (type 1) $r\ge l+1$, and $(r-l,c)$ is a removable corner of $\core(P)$,
		\item (type 2) $c\ge T+1$, and $(r,c-T)$ is a removable corner of $\mu$.
	\end{itemize}

	We put
	\begin{align*}
		X_j &:= \textrm{Res}\{\text{removable corners of $\core(P\cup\mu)$ of type $j$}\} \\
		Y_j &:= \textrm{Res}\{\text{$\core(P\cup\nu\cup(x))$-nonblocked removable corners of $\core(P\cup\mu)$ of type $j$}\}
	\end{align*}
	for $j=1,2$.

	We denote by $i$ the residue of top addable corner of $\core(P\cup \mu)$.
	Then we have
	$X_1 = \{i+t_1,i+t_2,\ldots,i+t_m\}$
	and 
	$i+\mu_l \in X_2 \subset [i+\mu_l, i+k-1]$.
	Note that 
	$$\{i+t_1,\ldots,i+t_m\} \cap [i+\mu_l, i+k-1] = 
	\begin{cases}
		\emptyset & \text{(if $t_m<\mu_l$)}, \\
		\{i+\mu_l\} & \text{(if $t_m=\mu_l$)}.
	\end{cases}$$

	Next we show that, for $i\ge 1$, 
	\begin{align}
		\core(P\cup\nu\cup(x))_{l+i}&=\core(P\cup(x))_i, \label{eq:Cl_two_former} \\
		\core(P\cup\nu\cup(x))'_{T + i} &= \core(\nu\cup(x))'_i. \label{eq:Cl_two_latter}
	\end{align}

	(\ref{eq:Cl_two_former}) is obvious since the smallest part of $\nu$,
	which is $\nu_l$, 
	is greater than or equal to the largest part of $P\cup(x)$,
	which is $\max\{x,t_m\}$.

	For (\ref{eq:Cl_two_latter}),
	first we note that, by (\ref{eq:Cl_two_former})
	and Figure \ref{fig:c_P_s} in the proof of Proposition \ref{Rt_prod_hr}, we have
	\begin{align*}
	\core(P\cup\nu\cup(x))_{l+1} &=\core(P\cup(x))_1 = T+x, \\
	\core(P\cup\nu\cup(x))_{l+2} &=\core(P\cup(x))_2 = T, \\
						&\vdots\\
		\core(P\cup\nu\cup(x))_{l+k+1-\mu_l} &=\core(P\cup(x))_{k+1-\mu_l} = T.
	\end{align*}

	Then by Lemma \ref{core_j} we have, for $1\le i \le l$,
	\begin{align*}
	\core(P\cup\nu\cup(x))_{i}
	&= \core(P\cup\nu\cup(x))_{i+(k+1-\nu_{i})} + \nu_i \\
	&= 
	\begin{cases}
		\core(P\cup\nu\cup(x))_{l+1} + \nu_1 = T+x+\nu_1 & \text{(if $i=1$ and $\nu_1=k+1-l$)}, \\
		T + \nu_i & \text{(otherwise)},\\
	\end{cases}
	\end{align*}
	where we used 
	$(P\cup\nu\cup(x))_{i} = \nu_i$ for $1\le i \le l$ for the first equality and
	$l+1\le i+(k+1-\nu_i) \le l+(k+1-\mu_n)$ (the first equality holds if and only if $i=1$ and $\nu_1=k+1-l$)
	for the second equality.

	Thus we have
	$\core(P\cup\nu\cup(x))_{i}=\core(\nu\cup(x))_{i}+T$ for $1\le i\le l+1$
	and
	$\core(P\cup\nu\cup(x))_{i}\le T$ for $i> l+1$,
	which implies
	(\ref{eq:Cl_two_latter}). 

		\vspace{2mm}
		Hence, 
		$|Y_1| = r_{\core(P\cup(x)),\core(P)} = \al{x+1}$, 
		and 
		$|Y_2| = r_{\core(\nu\cup(x)),\mu} = r_{\nu\mu}-\de\left[x\ge\mu_l\right]$.

		Moreover $Y_1\cap Y_2 = \{i+\mu_l\}$ if $x < t_m=\mu_l$, and $Y_1\cap Y_2=\emptyset$ otherwise. Then
	\begin{align*}
		r_{\core(P\cup\nu\cup(x)),\core(P\cup\mu)} &= |Y_1|+|Y_2|-|Y_1\cap Y_2| \\
		&=\begin{cases}
		\al{x+1} + r_{\nu\mu} - \de\left[x\ge\mu_l\right] & \text{(if $t_m<\mu_l$)}, \\
		\al{x+1} + r_{\nu\mu} \underbrace{- \de\left[x\ge\mu_l\right] - \de\left[x<t_m\right]}_{=-1} & \text{(if $t_m=\mu_l$)}.
		\end{cases}
	\end{align*}

\end{proof}

Thus Steps (A-1) and (B-1) have been achieved.

\subsection{Steps (A-2) and (B-2)}
The next lemma is technically important to perform the instructions in Step (A-2) and (B-2).

\begin{lemm}\label{binom_heikouidou}
Let $\nu, u, p$ be as in the assumptions in Definition \ref{def:T} and $n$ be an integer.
Then we have the following equalities.
In particular, in either case, the left-hand side does not depend on $p$.

$(1)$	$\DS\sum_{i=0}^{u} \binom{p+n+i-1}{i} T_{\nu,u-i,p} = \sum_{s=0}^{u} \binom{n+s-1}{s} \kks{\nu\cup(u-s)}$.

$(2)$	$\DS\sum_{i=0}^{u} \binom{p+n+i-1}{i} T'_{P,\nu,u-i,p} = \sum_{s=0}^{u} \binom{n-\al{u+1-s}+s-1}{s} \kks{P\cup \nu \cup (u-s)}$. 
\end{lemm}

\begin{proof}
	Since both equality can be proved in a parallel manner, we prove (2) here.
	By the definition of $T'_{P,\nu,u-i,p}$, we have
	\begin{align*}
		(\textrm{LHS}) &= \sum_{i=0}^{u} \binom{p+n+i-1}{i} \sum_{s=0}^{u-i} (-1)^s \binom{p+\al{u-i+1-s}}{s} \kks{P\cup\nu\cup(u-i-s)}, \\
\intertext{then putting $t=i+s$,}
		      &= \sum_{t = 0}^{u} \bigg( \sum_{s=0}^{t} \binom{p+n-1+t-s}{t-s} (-1)^s \binom{p+\al{u+1-t}}{s} \bigg) \kks{P\cup\nu\cup(u-t)}, \\
\intertext{then using Lemma \ref{binom_lem1},}
		      &= \sum_{t=0}^{u} (-1)^t \binom{-n+\al{u+1-t}}{t} \kks{P\cup\nu\cup(u-t)} \\
		      &= \sum_{t=0}^{u} \binom{n-\al{u+1-t}+t-1}{t} \kks{P\cup\nu\cup(u-t)}.
	\end{align*}
\end{proof}

Now we can express $\kks{\la}$ (resp.\ $\kks{P\cup\la}$)
as a linear combination of $\kks{\mu}h_r$ (resp.\ $\kks{P\cup\mu}h_r$) as proposed in the description of Step (A-2) (resp.\ (B-2)).

\begin{lemm}\label{g_mu_r_expand}
	Let $P$ and $\al{u}$ $(u \in \Z_{>0})$ be as in \NP.
	Let $\la,\bla,\bl$ be as in \Nla in Section \ref{sec:dist_ex}. 
	Write $r=\la_{l(\la)}$.
	Assume that $\bl\ge 1$ and $\maxti \le \bla_\bl$.
\begin{flalign*}
\intertext{$(1)$ We have}
&&	\kks{\la} &= \sum_{\substack{\mu \text{\rm \ s.t.}\\ \bar\la\subset\mu\subset \Rbl}} \sum_{q\in \mathbb{Z}}
	A_{\mu,\bar\la,q} \kks{\mu} \sum_{i\ge 0}
	\binom{q+i-1}{i} h_{r-|\mu/\bar\la|-i}. & \\
\intertext{$(2)$ If $\maxti <\bla_\bl$, we have} 
&&	\kks{P\cup\la} &= \sum_{\substack{\mu \text{\rm \ s.t.}\\ \bar\la\subset\mu\subset \Rbl}} \sum_{q\in \mathbb{Z}} 
	A_{\mu,\bar\la,q} \kks{P\cup\mu} \sum_{i\ge 0}
	\binom{q+i+\al{r}-1}{i} h_{r-|\mu/\bar\la|-i}. & \\
\intertext{$(3)$ If $\maxti =\bla_\bl$, we have}
&&	\kks{P\cup\la} &= \sum_{\substack{\mu \text{\rm \ s.t.}\\ \bar\la\subset\mu\subset \Rbl}} \sum_{q\in \mathbb{Z}}
	A_{\mu,\bar\la,q} \kks{P\cup\mu} \sum_{i\ge 0}
	\binom{q+i+\al{r}-2+\de\left[\mu_\bl\neq \bar\la_\bl\right]}{i} h_{r-|\mu/\bar\la|-i}. & 
\end{flalign*}
Here, in all of the three expressions, the number
$A_{\mu,\bar\la,q}$ is defined by the following recursion formula:
\begin{align*}
	A_{\bar\la,\bar\la,q} &= \delta_{q,r_{\bar\la\bar\la}}, \\
	A_{\mu,\bar\la,q} &= -\sum_{\substack{\mu/\ka\text{\rm: h.s.} \\ \bar\la\subset\ka\subsetneq\mu}} A_{\ka,\bar\la,q-(r_{\mu\mu}-r_{\mu\ka})} \ \ \text{for $\bar\la\subsetneq\mu\subset\Rbl$}.
\end{align*}
Notice that for each $\mu$, $A_{\mu,\bar\la,q} = 0$ except for finitely many $q$.
The explicit value of $A_{\mu,\bar\la,q}$ will be given in Lemma \ref{A_la_q} below.
\end{lemm}

\noindent{\it Remark.}
In the above recursion formula, $r_{\mu\mu} - r_{\mu\ka} \ge 0$ always holds
because there must be a $\mu$-removable corner in every row in which there is a $\mu$-nonblocked $\ka$-removable corner
since $\mu/\ka$ is a horizontal strip.

\begin{proof}
(1) By Lemma \ref{kks_prod_hr}(1), 
	\begin{align*}
		(\textrm{RHS}) &= \sum_{\substack{\mu \text{ s.t.}\\ \bar\la\subset\mu\subset \Rbl}} \sum_{q} A_{\mu,\bar\la,q} \sum_{\substack{\mu\subset\nu\subset \Rbl \\ \nu/\mu\text{: h.s.}}} \sum_{i\ge 0}
		\binom{q+i-1}{i} T_{\nu,r-|\mu/\bar\la|-i-|\nu/\mu|,r_{\nu\mu}}, \\
\intertext{then splitting the third summation according to whether $\mu=\nu$ or $\mu\subsetneq\nu$,}
			&= \sum_{\substack{\mu \text{ s.t.}\\ \bar\la\subset\mu\subset \Rbl}} \sum_{q} A_{\mu,\bar\la,q} \sum_{i\ge 0}
			\binom{q+i-1}{i} T_{\mu,r-|\mu/\bar\la|-i,r_{\mu\mu}} \\
			&\phantom{=}+ \sum_{\substack{\mu,\nu \text{ s.t.}\\ \bar\la\subset\mu\subsetneq\nu\subset \Rbl \\ \nu/\mu\text{: h.s.} }} \sum_{q} A_{\mu,\bar\la,q} \sum_{i\ge 0}
			\binom{q+i-1}{i} T_{\nu,r-|\nu/\bar\la|-i,r_{\nu\mu}}, \\
\intertext{then replacing the variable $\mu$ for the first summation by $\nu$, and splitting it again according to whether $\bar\la=\nu$ or $\bar\la\subsetneq\nu$, and rearranging the summands,}
		      &= \sum_{q} A_{\bar\la,\bar\la,q} \sum_{i\ge 0} \binom{q+i-1}{i} T_{\bar\la,r-i,r_{\bar\la\bar\la}} \\
				 &\phantom{=}+\sum_{\substack{\nu \text{ s.t.} \\ \bar\la\subsetneq\nu\subset\Rbl}}
			\bigg(\underbrace{\sum_{q}A_{\nu,\bar\la,q}\sum_{i\ge 0} \binom{q+i-1}{i} T_{\nu,r-|\nu/\bar\la|-i,r_{\nu\nu}}}_{(X)} \\
			&\phantom{+\sum_{\substack{\nu \text{ s.t.} \\ \bar\la\subsetneq\nu\subset\Rl}}}
			+ \underbrace{\sum_{q} \sum_{\substack{\mu\text{ s.t.}\\ \bar\la\subset\mu\subsetneq\nu \\ \nu/\mu\text{: h.s.} }}
			A_{\mu,\bar\la,q} \sum_{i\ge 0} 
			\binom{q+i-1}{i} T_{\nu,r-|\nu/\bar\la|-i,r_{\nu\mu}} }_{(Y)} \bigg).
	\end{align*}
Then, by the definition of $A_{\nu,\bar\la,q}$, noting that $\bar\la\subsetneq\nu$, 
\begin{align*}
	(X) &= - \sum_{q} \sum_{\substack{\mu \text{ s.t.} \\ \nu/\mu\text{: h.s.} \\ \bar\la\subset\mu\subsetneq\nu }} A_{\mu,\bar\la,q-(r_{\nu\nu}-r_{\nu\mu})}
		\sum_{i \ge 0} \binom{q+i-1}{i} T_{\nu,r-|\nu/\bar\la|-i,r_{\nu\nu}}, \\
\intertext{then replacing $q$ by $q+(r_{\nu\nu}-r_{\nu\mu})$,}
	&= -\sum_{q} \sum_{\substack{\mu \text{ s.t.} \\ \nu/\mu\text{: h.s.} \\ \bar\la\subset\mu\subsetneq\nu }} A_{\mu,\bar\la,q}
		\sum_{i \ge 0} \binom{q+r_{\nu\nu}-r_{\nu\mu}+i-1}{i} T_{\nu,r-|\nu/\bar\la|-i,r_{\nu\nu}}, \\
\intertext{then using the independence of the LHS on $p$ of Lemma \ref{binom_heikouidou}(1)
	(note that the range of $i$ can be limited to $0\le i \le r-|\nu/\bla|$ since $i$ originally occurs in $h_{r-|\mu/\bla|-i}$
	in the statement of part (1) of the Lemma),}
	&= -\sum_{q} \sum_{\substack{\mu \text{ s.t.} \\ \nu/\mu\text{: h.s.} \\ \bar\la\subset\mu\subsetneq\nu }} A_{\mu,\bar\la,q}
		\sum_{i \ge 0} \binom{q+i-1}{i} T_{\nu,r-|\nu/\bar\la|-i,r_{\nu\mu}} \\
	    &= -(Y).
\end{align*}
Hence,
\begin{align*}
	(\textrm{RHS}) 	&= \sum_{q} A_{\bar\la,\bar\la,q} \sum_{i\ge 0} \binom{q+i-1}{i} T_{\bar\la,r-i,r_{\bar\la\bar\la}} \\
		&= \sum_{i\ge 0} \binom{r_{\bar\la\bar\la}+i-1}{i} T_{\bar\la,r-i,r_{\bar\la\bar\la}}, \\
\intertext{again by Lemma \ref{binom_heikouidou}(1), noting that $\binom{0+s-1}{s}$ vanishes unless $s=0$,}
		&= \kks{\la}.
\end{align*}

\noindent
(3) is proved almost parallel to (1):
By Lemma \ref{kks_prod_hr}(2) and (3),
\begin{align*}
	(\textrm{RHS}) &= \sum_{\substack{\mu\text{ s.t.}\\ \bar\la\subset\mu\subset \Rbl}} 
	 \sum_{q} A_{\mu,\bar\la,q} 
	 \sum_{\substack{\mu\subset\nu\subset \Rbl \\ \nu/\mu\text{: h.s.}}} \\
	&\qquad\qquad
	 \sum_{i\ge 0}
	 \binom{q+i+\al{r}-2+\de\left[\mu_l\neq\bar\la_l\right]}{i} T'_{P,\nu,r-|\mu/\bar\la|-i-|\nu/\mu|,r_{\nu\mu}-\de\left[\mu_l=\bar\la_l\right]}, \\
\intertext{then, by Lemma \ref{binom_heikouidou}(2), shifting $p$ by $\DE{\mu_l=\bla_\bl}$
	and noting that $\DE{\mu_l\neq\bla_\bl} + \DE{\mu_l=\bla_\bl} = 1$,}
&= \sum_{\substack{\mu\text{ s.t.}\\ \bar\la\subset\mu\subset \Rbl}}
	   \sum_{q} A_{\mu,\bar\la,q} 
	   \sum_{\substack{\mu\subset\nu\subset \Rbl \\ \nu/\mu\text{: h.s.}}} 
	   \sum_{i\ge 0}
	\binom{q+i+\al{r}-1}{i} T'_{P,\nu,r-|\nu/\bar\la|-i,r_{\nu\mu}}. \\
\intertext{Note that the following deformation is also valid for the case $\maxti<\bla_\bl$.
	Applying the same argument as (1),}
&= \sum_{i\ge 0} \binom{r_{\bar\la\bar\la}+\al{r}-1+i}{i} T'_{P,\bar\la,r-i,r_{\bar\la\bar\la}}, \\
\intertext{then by Lemma \ref{binom_heikouidou}(2),}
&= \sum_{s\ge 0} \binom{-\al{r+1-s}+\al{r}-1+s}{s} \kks{P\cup\bar\la\cup(r-s)} \\
	&= \kks{P\cup\la}.
\end{align*}
Here the last equality follows from $\binom{-\al{r+1-s}+\al{r}-1+s}{s}=(-1)^s\binom{\al{r+1-s}-\al{r}}{s}$ and 
$0\le \al{r+1-s}-\al{r} \le s-1$ for $s\ge 1$.

\vspace{2mm}
For (2), we have
\begin{align*}
	(\textrm{RHS}) &= \sum_{\substack{\mu\text{ s.t.}\\ \bar\la\subset\mu\subset \Rbl}} 
	 \sum_{q} A_{\mu,\bar\la,q}
	 \sum_{\substack{\mu\subset\nu\subset \Rbl \\ \nu/\mu\text{: h.s.}}} 
	 \sum_{i\ge 0}
	\binom{q+i+\al{r}-1}{i} T'_{P,\nu,r-|\mu/\bar\la|-i-|\nu/\mu|,r_{\nu\mu}},
\end{align*}
which is equal to $\kks{P\cup\la}$
since this sum has exactly the same form as appeared in the proof of (3).
\end{proof}

In fact we can explicitly solve the recursion formula of $A_{\mu,\bar\la,q}$ appeared in the previous proposition.
This result is needed in the author's following paper \cite{Takigiku} and included in Appendix \ref{sec:apdx_A_la_q}.

\vspace{2mm}

Now Step (A) and (B) have been accomplished.

\subsection{Step (C)}
We multiply $\kks{\la}$ by $\kks{P}$, and express it as a linear combination of $K$-$k$-Schur functions,
and solve it:

\begin{theo}\label{P_mu_r} 
	Let $P$ and $\al{u}$ $($for $u\in\Z_{>0})$ be as in \NP in Section \ref{sec:general}, before Proposition \ref{P_factor}.
	Let $\la,\bla,\bl$ be as in \Nla in Section \ref{sec:dist_ex}.
	Write $r=\la_{l(\la)}$.
	Assume $\max_i\{t_i\} < \bla_\bl$.
	Then we have
	\begin{flalign*}
	(1) &&
		\kks{P}\kks{\la} &= \sum_{s=0}^{r} (-1)^s \binom{\al{r+1-s}}{s} \kks{P\cup\bar\la\cup(r-s)}. & \\
	(2) &&
		\kks{P\cup\la} &= \kks{P} \sum_{s=0}^{r} \binom{\al{r}+s-1}{s} \kks{\bar\la\cup(r-s)}. &
	\end{flalign*}
	In particular, if $t_n < r$ then $\al{r}=0$ and
	\[
		\kks{P\cup\la} = \kks{P} \kks{\la}.
	\]
\end{theo}

\begin{proof}
	(2) follows from (1) and Lemma \ref{binom_inv_matrix}.
	We prove (1) by induction on $\bl\ge 0$.
	The case $\bl=0$ was proved in Proposition \ref{Rt_prod_hr} and Theorem \ref{R_t_r}.
	Assume $\bl\ge 1$.

	From Lemma \ref{g_mu_r_expand},
	\begin{align*}
		(\textrm{LHS}) &= \kks{P} \sum_{\substack{\mu\text{ s.t.}\\ \bar\la\subset\mu\subset \Rbl}} \sum_{q} A_{\mu,\bar\la,q} \kks{\mu} \sum_{i\ge 0} \binom{q+i-1}{i} h_{r-|\mu/\bar\la|-i}, \\
\intertext{by the induction hypothesis, we have $\kks{P}\kks{\mu}=\kks{P\cup\mu}$ in the above summation. Hence}
&= \sum_{\substack{\mu\text{ s.t.}\\ \bar\la\subset\mu\subset \Rbl}} \sum_{q} A_{\mu,\bar\la,q} \kks{P\cup\mu} \sum_{i\ge 0} \binom{q+i-1}{i} h_{r-|\mu/\bar\la|-i}, \\
\intertext{then by Lemma \ref{kks_prod_hr}(2), (notice that $\mu_\bl\ge\bla_\bl>\maxti$)}
&=\sum_{\substack{\mu\text{ s.t.}\\ \bar\la\subset\mu\subset \Rbl}} \sum_{q} A_{\mu,\bar\la,q} \sum_{\substack{\mu\subset\nu\subset \Rbl \\ \nu/\mu\text{: h.s.}}} \sum_{i\ge 0}
		\binom{q+i-1}{i} T'_{P,\nu,r-|\nu/\bar\la|-i,r_{\nu,\mu}}, \\
\intertext{then by doing the same argument as Lemma \ref{g_mu_r_expand}(1),
(formally replacing $T_{\dots}$ by $T'_{P,\dots}$ and using Lemma \ref{binom_heikouidou}(2) instead of (1),
the proof works)}
		&=\sum_{i\ge 0} \binom{r_{\bar\la\bar\la}+i-1}{i} T'_{P,\bar\la,r-i,r_{\bar\la\bar\la}}, \\
\intertext{then by Lemma \ref{binom_heikouidou}(2),}
		      &= \sum_{s=0}^{r} \binom{-\al{r+1-s}+s-1}{s} \kks{P\cup\bar\la\cup(r-s)} \\
		      &= \sum_{s=0}^{r} (-1)^s \binom{\al{r+1-s}}{s} \kks{P\cup\bar\la\cup(r-s)}.
	\end{align*}
\end{proof}

Now we can achieve
our goal in this section.

\begin{theo}\label{dist_krec_factor}
	For $1\le t_1 < \cdots < t_m \le k$
	and $a_1,\dots,a_m>0$, 
\[
	\kks{R_{t_1}^{a_1}\cup\cdots\cup R_{t_m}^{a_m}} 
	= \kks{R_{t_1}^{a_1}}\cdots\kks{R_{t_m}^{a_m}}.
\]
\end{theo}

\begin{proof}
Use induction on $m>0$.

The base case $m=1$ is obvious.
Assume $m>1$.

Applying Proposition \ref{Rt_twice} for 
$\la=R_{t_m}^{i}$ and $t=t_m$, we have
\[ \kks{R_{t_m}^{i+2}} = \kks{R_{t_m}^{i+1}} \frac{\kks{R_{t_m}\cup R_{t_m} }}{\kks{R_{t_m}}}. \]
Multiplying this for $i=0,\dots,a_m-2$, we have
\begin{equation}\label{eq:Rta}
\kks{R_{t_m}^{a_m}} = \kks{R_{t_m}} \left(\frac{\kks{R_{t_m}\cup R_{t_m} }}{\kks{R_{t_m}}}\right)^{a_m-1}. 
\end{equation}

Put $P = R_{t_1}^{a_1}\cup\cdots\cup R_{t_{m-1}}^{a_{m-1}}$.

Similarly applying Proposition \ref{Rt_twice} for 
$\la=P\cup R_{t_m}^{i}$ and $t=t_m$, 
then multiplying this for $i=0,\dots,a_m-2$, we have
\[ \kks{P\cup R_{t_m}^{a_m}} = \kks{P\cup R_{t_m}} \left(\frac{\kks{R_{t_m}\cup R_{t_m} }}{\kks{R_{t_m}}}\right)^{a_m-1}. \]

On the other hand, applying the previous theorem for $P$, $\la = R_{t_m}$, 
we have $\kks{P \cup R_{t_m}} = \kks{P} \kks{R_{t_m}}$.

Hence we have 
$$ \kks{P\cup R_{t_m}^{a_m}}
= \kks{P} \kks{R_{t_m}} \left(\frac{\kks{R_{t_m}\cup R_{t_m} }}{\kks{R_{t_m}}}\right)^{a_m-1}
= \kks{P} \kks{R_{t_m}^{a_m}} 
= \kks{R_{t_1}^{a_1}} \dots \kks{R_{t_m}^{a_m}},$$
where the last equality follows by the induction hypothesis.
\end{proof}

\section{Discussions}\label{sec:discuss}

In this section we state some conjectures, that are consistent with our results in previous sections.

\begin{conj}
	For all $\la \in \Pk$ and $P=\RRma$, write
	$$\kks{P\cup\la} = \kks{P} \sum_{\mu} a_{P,\la,\mu} \kks{\mu}.$$
	Then $a_{P,\la,\mu}\ge 0$ for any $\mu$.
\end{conj}

In the case $P=R_t$, it is observed that $a_{R_t,\la,\mu}=0$ or $1$.
Moreover, the set of $\mu$ such that $a_{R_t,\la,\mu}=1$ is expected to be an ``interval'',
but we have to consider the \textit{strong order} on $\Pk\simeq\Cn\simeq\tSn$,
which can be seen as just inclusion as shapes in the poset of cores,
or strong Bruhat order on the affine symmetric group.
Namely, the strong order $\la\le\mu$ on $\Pk$ is defined by $\core(\la)\subset\core(\mu)$.
Notice that $\la\preceq\mu \Lra \la\subset\mu \Lra \la\le\mu$ for $\la,\mu\in\Pk$.
Then,

\begin{conj}\label{conj:interval}
	For all $\la \in \Pk$ and $1\le t \le k$, 
	there exists $\mu \in \Pk$ such that
	\[
		\kks{R_t\cup \la} = \kks{R_t} \sum_{\mu\le\nu\le\la} \kks{\nu}.
	\]
\end{conj}
	Assuming this conjecture, we shall write
	$\minindex(\la,t) = \mu$.

We can make some conjectures about the behavior of minindex:
\begin{itemize}
\item
	It is expected that if $\la$ gets ``bigger'' with respect to inclusion, then minindex gets weakly bigger in the strong order.
	Namely,
	
\noindent\textit{
	For any two elements $\mu \subset\la$ of $\Pk$, we have
	$\minindex(\mu,t) \le \minindex(\la,t)$.
}
\item
	If a bounded partition has a form $R_s\cup\la$ for $s\neq t$, 
	its minindex still be expected to contains $R_s$,
	and the ``remaining part'' is bigger or equal to $\minindex(\la,t)$ in the strong order:

\noindent\textit{
	For all $\la\in\Pk$ and $1\le s \neq t \le k$, 
	$\minindex(R_s\cup\la,t)$ has the form $R_s\cup\mu$ and
	$\minindex(\la,t) \le \mu$.
}
\item
	If a bounded partition has a form $R_s\cup R_s\cup\la$ for $s\neq t$, 
	its minindex would be equal to the union of $R_s$ and $\minindex(R_s\cup\la)$:

\noindent\textit{
	For all $\la\in\Pk$ and $1\le s\neq t\le k$, we have
	$R_s\cup\minindex(R_s\cup\la,t) = \minindex(R_s\cup R_s \cup \la, t)$.
}
\end{itemize}

\vspace{2mm}
Next, consider a bounded partition that has a form $R_t\cup\la$.
We wrote
\[
\frac{\kks{R_t\cup R_t\cup \la}}{\kks{R_t}} = \sum_{\minindex(R_t\cup\la,t) \le \ga \le R_t\cup\la} \kks{\ga}.
\]
On the other hand, by Proposition \ref{Rt_twice}
\begin{align*}
\frac{\kks{R_t\cup R_t\cup\la}}{\krt}
&= \frac{\krtl}{\krt} \frac{\kks{R_t\cup R_t}}{\krt} 
= \sum_{\minindex(\la,t)\le\mu\le\la} \kks{\mu} \sum_{\nu\subset R_t} \kks{\nu}.
\end{align*}
Now we can expect that 
for any $\mu\in\Pk$, 
\[
\kks{\mu} \sum_{\nu\subset R_t} \kks{\nu} = \sum_{\ga\in I_{\mu,t}} \kks{\ga},
\]
where $I_{\mu,t}$ is an order filter of the interval $[\emptyset, R_t\cup\mu]$ (in $\Pk$ with the strong order)
such that 
\[
\bigsqcup_{\mu\in[\minindex(\la,t),\la]} I_{\mu,t} = [\minindex(R_t\cup\la,t), R_t\cup\la].
\]

\appendix

\section{Examples} \label{sec:example}

In this section we sometimes abbreviate $\sum_{\la} a_{\la} \kks[3]{\la}$ as $\sum_{\la} a_{\la} \la$
for ease to see.

\vspace{2mm}

\begin{table}[H]
	\caption{$k=3$. The table of ${\kks[3]{Q\cup\la}}/{\kks[3]{Q}}$
		for $Q=R_{t_1}\cup\dots\cup R_{t_n}$ ($1\le t_1<\dots<t_n\le k$),
		$\la\subset (1^2 2^1 3^0)$
	}
	\label{table:Q_la}
	\centering
\scalebox{0.8}{
\begin{tabular}{|c||c|c|c|c|c|} \hline
	\backslashbox{$Q$}{$\la$}&
	$\syou{1}$ &
	$\syou{1,1}$ &
	$\syou{2}$ &
	$\syou{1,2}$ &
		$\syou{1,1,2}$
	\\ \hline \hline
	$\syou{3}$ &
	$\syou{1}+\emptyset$ &
	$\syou{1,1}$ &
	$\syou{2}+\syou{1}+\emptyset$ &
	$\syou{1,2}+\syou{2}$ &
	\begin{minipage}{4cm}
		\vspace{1mm}
	$\syou{1,1,2}+\syou{1,1,1}+\syou{1,2}+\syou{1,1}$
		\vspace{1mm}
	\end{minipage}
	\\ \hline
	$\syou{2,2}$ &
	$\syou{1}+\emptyset$ &
	$\syou{1,1}+\syou{1}+\emptyset$ &
	$\syou{2}+\syou{1}+\emptyset$ &
	$\syou{1,2}+\syou{2}+\syou{1,1}+\syou{1}+\emptyset$ &
	\begin{minipage}{4cm}
		\vspace{1mm}
		$\syou{1,1,2}+\syou{1,1,1}+\syou{1,2}+\syou{3}+\syou{1,1}+\syou{2}+\syou{1}+\emptyset$
		\vspace{1mm}
	\end{minipage}
	\\ \hline
	\raisebox{-1ex}{$\syou{1,1,1}$} &
	$\syou{1}+\emptyset$ &
	$\syou{1,1}+\syou{1}+\emptyset$ &
	$\syou{2}$ &
	$\syou{1,2}+\syou{1,1}$ &
	\begin{minipage}{4cm}\vspace{1mm}
	$\syou{1,1,2}+\syou{1,2}+\syou{3}+\syou{2}$
	\vspace{1mm}
	\end{minipage}
	\\ \hline
	\raisebox{-1ex}{$\syou{2,2,3}$} &
	$\syou{1}+2\emptyset$ &
	$\syou{1,1}+\syou{1}+\emptyset$ &
	$\syou{2}+2\syou{1}+3\emptyset$ &
	$\syou{1,2}+\syou{2}+2\syou{1,1}+2\syou{1}+2\emptyset$ &
	\begin{minipage}{4cm}\vspace{1mm}
		$\syou{1,1,2}+2\syou{1,1,1}+2\syou{1,2}+\syou{3}+3\syou{1,1}+2\syou{2}+3\syou{1}+3\emptyset$\vspace{1mm}
	\end{minipage}
	\\ \hline
	\raisebox{-2ex}{$\syou{1,1,1,3}$} &
	$\syou{1}+2\emptyset$ &
	$\syou{1,1}+\syou{1}+\emptyset$ &
	$\syou{2}+\syou{1}+\emptyset$ &
	$\syou{1,2}+\syou{2}+\syou{1,1}+\syou{1}+\emptyset$ &
	\begin{minipage}{4cm}\vspace{1mm}
		$\syou{1,1,2}+2\syou{1,1,1}+2\syou{1,2}+\syou{3}+2\syou{1,1}+2\syou{2}+2\syou{1}+2\emptyset$\vspace{1mm}
	\end{minipage}
	\\ \hline
	\raisebox{-2ex}{$\syou{1,1,1,2,2}$} &
	$\syou{1}+2\emptyset$ &
	$\syou{1,1}+2\syou{1}+3\emptyset$ &
	$\syou{2}+\syou{1}+\emptyset$ &
	$\syou{1,2}+2\syou{2}+\syou{1,1}+2\syou{1}+2\emptyset$ &
	\begin{minipage}{4cm}\vspace{1mm}
		$\syou{1,1,2}+\syou{1,1,1}+2\syou{1,2}+2\syou{3}+2\syou{1,1}+2\syou{2}+3\syou{1}+3\emptyset$\vspace{1mm}
	\end{minipage}
	\\ \hline
	\raisebox{-3ex}{$\syou{1,1,1,2,2,3}$} &
	$\syou{1}+3\emptyset$ &
	$\syou{1,1}+2\syou{1}+3\emptyset$ &
	$\syou{2}+2\syou{1}+3\emptyset$ &
	$\syou{1,2}+2\syou{2}+2\syou{1,1}+4\syou{1}+5\emptyset$ &
	\begin{minipage}{4cm}\vspace{2mm}
		$\syou{1,1,2}+2\syou{1,1,1}+3\syou{1,2}+2\syou{3}+5\syou{1,1}+5\syou{2}+8\syou{1}+9\emptyset$\vspace{2mm}
	\end{minipage}
	\\ \hline
\end{tabular}
}
\end{table}

\begin{table}[H]
	\label{table:R_three_la}
	\caption{$k=3$. The table of $\minindex(\la,t)$ for $|\la|\le 6$.}
	\centering

	\begin{tabular}{cccccccc} \hline
	$\la$ & $\core(\la)$ & \multicolumn{6}{c}{$\minindex$, $\core(\minindex)$}\\ \cline{3-8}
		&	& \multicolumn{2}{c}{$t=1$} & \multicolumn{2}{c}{$t=2$} & \multicolumn{2}{c}{$t=3$} \\ \hline
	$\emptyset$ & $\emptyset$ & $\emptyset$ & $\emptyset$ & $\emptyset$ & $\emptyset$ & $\emptyset$ & $\emptyset$ \\
	$\syou{1}$ & $\syou{1}$ & $\emptyset$ & $\emptyset$ & $\emptyset$ & $\emptyset$ & $\emptyset$ & $\emptyset$ \\
	$\syou{2}$ & $\syou{2}$ & $\syou{2}$ & $\syou{2}$ & $\emptyset$ & $\emptyset$ & $\emptyset$ & $\emptyset$ \\
	$\syou{1,1}$ & $\syou{1,1}$ & $\emptyset$ & $\emptyset$ & $\emptyset$ & $\emptyset$ & $\syou{1,1}$ & $\syou{1,1}$ \\
	$\syou{3}$ & $\syou{3}$ & $\syou{3}$ & $\syou{3}$ & $\syou{3}$ & $\syou{3}$ & $\emptyset$ & $\emptyset$ \\
	$\syou{1,2}$ & $\syou{1,2}$ & $\syou{2}$ & $\syou{2}$ & $\emptyset$ & $\emptyset$ & $\syou{2}$ & $\syou{2}$ \\
	$\syou{1,1,1}$ & $\syou{1,1,1}$ & $\emptyset$ & $\emptyset$ & $\syou{1,1,1}$ & $\syou{1,1,1}$ & $\syou{1,1,1}$ & $\syou{1,1,1}$ \\
	$\syou{1,3}$ & $\syou{1,4}$ & $\syou{3}$ & $\syou{3}$ & $\syou{3}$ & $\syou{3}$ & $\emptyset$ & $\emptyset$ \\
	$\syou{2,2}$ & $\syou{2,2}$ & $\syou{2,2}$ & $\syou{2,2}$ & $\emptyset$ & $\emptyset$ & $\syou{2,2}$ & $\syou{2,2}$ \\
	$\syou{1,1,2}$ & $\syou{1,1,3}$ & $\syou{2}$ & $\syou{2}$ & $\emptyset$ & $\emptyset$ & $\syou{1,1}$ & $\syou{1,1}$ \\
	$\syou{1,1,1,1}$ & $\syou{1,1,1,2}$ & $\emptyset$ & $\emptyset$ & $\syou{1,1,1}$ & $\syou{1,1,1}$ & $\syou{1,1,1}$ & $\syou{1,1,1}$ \\ 
	$\syou{2,3}$ & $\syou{2,5}$ & $\syou{2,3}$ & $\syou{2,5}$ & $\syou{3}$ & $\syou{3}$ & $\emptyset$ & $\emptyset$ \\
	$\syou{1,1,3}$ & $\syou{1,1,4}$ & $\syou{1,3}$ & $\syou{1,4}$ & $\syou{1,3}$ & $\syou{1,4}$ & $\syou{1,1,1}$ & $\syou{1,1,1}$ \\
	$\syou{1,2,2}$ & $\syou{1,2,3}$ & $\syou{2,2}$ & $\syou{2,2}$ & $\emptyset$ & $\emptyset$ & $\syou{2,2}$ & $\syou{2,2}$ \\
	$\syou{1,1,1,2}$ & $\syou{1,1,1,3}$ & $\syou{3}$ & $\syou{3}$ & $\syou{1,1,1,1}$ & $\syou{1,1,1,2}$ & $\syou{1,1,1,1}$ & $\syou{1,1,1,2}$ \\
	$\syou{1,1,1,1,1}$ & $\syou{1,1,1,2,2}$ & $\emptyset$ & $\emptyset$ & $\syou{1,1,1}$ & $\syou{1,1,1}$ & $\syou{1,1,1,1,1}$ & $\syou{1,1,1,2,2}$ \\
	$\syou{3,3}$ & $\syou{3,6}$ & $\syou{3,3}$ & $\syou{3,6}$ & $\syou{3,3}$ & $\syou{3,6}$ & $\emptyset$ & $\emptyset$ \\
	$\syou{1,2,3}$ & $\syou{1,2,5}$ & $\syou{2,3}$ & $\syou{2,5}$ & $\syou{1,3}$ & $\syou{1,4}$ & $\syou{1,1,1}$ & $\syou{1,1,1}$ \\
	$\syou{2,2,2}$ & $\syou{2,2,4}$ & $\syou{2,2,2}$ & $\syou{2,2,4}$ & $\emptyset$ & $\emptyset$ & $\syou{2,2}$ & $\syou{2,2}$ \\
	$\syou{1,1,1,3}$ & $\syou{1,1,1,4}$ & $\syou{1,3}$ & $\syou{1,4}$ & $\syou{1,1,1,3}$ & $\syou{1,1,1,4}$ & $\syou{1,1,1,1}$ & $\syou{1,1,1,2}$ \\
	$\syou{1,1,2,2}$ & $\syou{1,1,3,3}$ & $\syou{2,2}$ & $\syou{2,2}$ & $\emptyset$ & $\emptyset$ & $\syou{1,1,2,2}$ & $\syou{1,1,3,3}$ \\
	$\syou{1,1,1,1,2}$ & $\syou{1,1,1,2,3}$ & $\syou{3}$ & $\syou{3}$ & $\syou{1,1,1,1}$ & $\syou{1,1,1,2}$ & $\syou{1,1,1,1,1}$ & $\syou{1,1,1,2,2}$ \\
	$\syou{1,1,1,1,1,1}$ & $\syou{1,1,1,2,2,2}$ & $\emptyset$ & $\emptyset$ & $\syou{1,1,1,1,1,1}$ & $\syou{1,1,1,2,2,2}$ & $\syou{1,1,1,1,1,1}$ & $\syou{1,1,1,2,2,2}$ \\
	\hline
	\end{tabular}
\end{table}

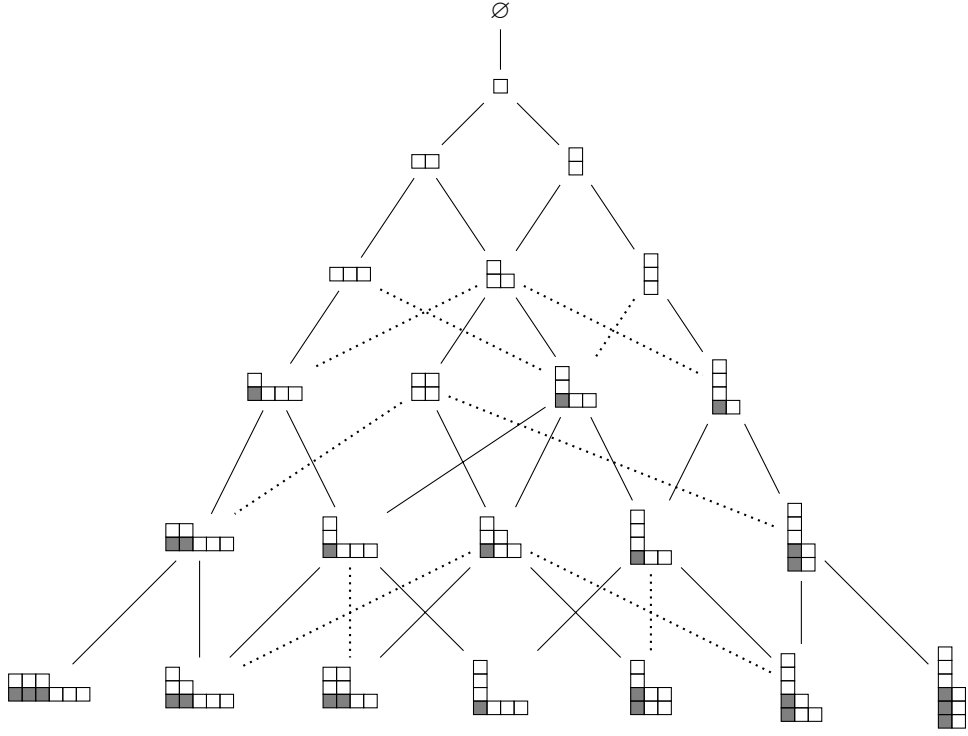
\begin{figure}[H]
\begin{tikzpicture}[scale=1]
	\node (0) at (0,0) {$\emptyset$};
	\node (1) at (0,-1) {$\sk{1}{}$};
	\node (2) at (-1,-2) {$\sk{2}{}$};
	\node (11) at (1,-2) {$\sk{1,1}{}$};
	\node (3) at (-2,-3.5) {$\sk{3}{}$};
	\node (21) at (0,-3.5) {$\sk{2,1}{}$};
	\node (111) at (2,-3.5) {$\sk{1,1,1}{}$};
	\node (31) at (-3,-5) {$\sk{4,1}{1}$};
	\node (22) at (-1,-5) {$\sk{2,2}{}$};
	\node (211) at (1,-5) {$\sk{3,1,1}{1}$};
	\node (1111) at (3,-5) {$\sk{2,1,1,1}{1}$};
	\node (32) at (-4,-7) {$\sk{5,2}{2}$};
	\node (311) at (-2,-7) {$\sk{4,1,1}{1}$};
	\node (221) at (0,-7) {$\sk{3,2,1}{1}$};
	\node (2111) at (2,-7) {$\sk{3,1,1,1}{1}$};
	\node (11111) at (4,-7) {$\sk{2,2,1,1,1}{1,1}$};
	\node (33) at (-6,-9) {$\sk{6,3}{3}$};
	\node (321) at (-4,-9) {$\sk{5,2,1}{2}$};
	\node (222) at (-2,-9) {$\sk{4,2,2}{2}$};
	\node (3111) at (0,-9) {$\sk{4,1,1,1}{1}$};
	\node (2211) at (2,-9) {$\sk{3,3,1,1}{1,1}$};
	\node (21111) at (4,-9) {$\sk{3,2,1,1,1}{1,1}$};
	\node (111111) at (6,-9) {$\sk{2,2,2,1,1,1}{1,1,1}$};
	\draw (0) -- (1)
		(1) -- (2)
		(1) -- (11)
		(2) -- (3)
		(2) -- (21)
		(11) -- (21)
		(11) -- (111)
		(3) -- (31)
		(21) -- (22)
		(21) -- (211)
		(111) -- (1111)
		(31) -- (32)
		(31) -- (311)
		(22) -- (221)
		(211) -- (311)
		(211) -- (221)
		(211) -- (2111)
		(1111) -- (2111)
		(1111) -- (11111)
		(32) -- (33)
		(32) -- (321)
		(311) -- (321)
		(311) -- (3111)
		(221) -- (222)
		(221) -- (2211)
		(2111) -- (3111)
		(2111) -- (21111)
		(11111) -- (21111)
		(11111) -- (111111);
	\draw[thick, dotted]
		(3) -- (211)
		(21) -- (31)
		(21) -- (1111)
		(111) -- (211)
		(22) -- (32)
		(22) -- (11111)
		(311) -- (222)
		(221) -- (321)
		(221) -- (21111)
		(2111) -- (2211)
		;
\end{tikzpicture}

\caption{The poset of $4$-cores (up to those of size $6$). The weak cover relations correspond to the solid lines, and the strong cover relations correspond to the solid or dotted lines.}
\label{fig:poset_core}
\end{figure}

\section{Proof of Proposition \ref{weakstrip}}\label{sec:apdx_ws}
\noindent\underline{$(4) \Lra (1)$}:
The latter condition
$\tau \wcover \exists \tau^{(1)} \wcover \dots \wcover \exists \tau^{(r)} = \ka$
is obvious.

If $\ka/\tau$ is not a horizontal strip, 
then $(a,b),(a+1,b)\in\ka/\tau$ for $\exists a,b$.
Write their residues $i=b-a$, $i-1=b-(a+1)$.
Then a $s_{i-1}$-action should be performed after a $s_i$-action.

Then the representation of (4) has the form
$\ka = \dots s_{i-1} \dots s_i \dots \tau$,
which contradicts (4).

\vspace{2mm}
\noindent\underline{$(1) \Lra (4)$}:
Assume 
$\ka = \dots s_i \dots s_{i+1} \dots \tau$,
$\ksize{\ka} = \ksize{\tau} + r$,
and $\ka/\tau$ is a horizontal strip.

Consider the moment
just before performing the action of $s_{i+1}$.
At that time
the situation around each extremal cell of residue $i+1$ is
one of the following:

(1):
\raisebox{-4ex}
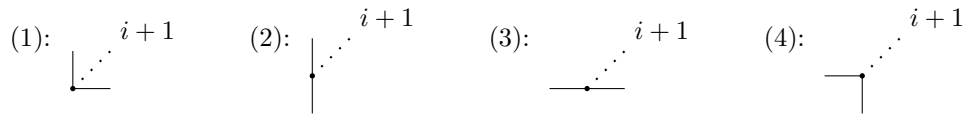
{
\begin{tikzpicture}[scale=0.5]
\draw (0,1) |- (1,0);
\fill (0,0) circle (2pt);
\draw [loosely dotted, thick] (0,0) -- (1.0,1.0) node [anchor=south west] {$i+1$};
\end{tikzpicture}
}
\hspace{6mm}
(2):
\raisebox{-6ex}{
\begin{tikzpicture}[scale=0.5]
\draw (0,1) -- (0,-1);
\fill (0,0) circle (2pt);
\draw [loosely dotted, thick] (0,0) -- (1.0,1.0) node [anchor=south west] {$i+1$};
\end{tikzpicture}
}
\hspace{6mm}
(3):
\raisebox{-4ex}{
\begin{tikzpicture}[scale=0.5]
\draw (-1,0) -- (1,0);
\fill (0,0) circle (2pt);
\draw [loosely dotted, thick] (0,0) -- (1.0,1.0) node [anchor=south west] {$i+1$};
\end{tikzpicture}
}
\hspace{6mm}
(4):
\raisebox{-6ex}{
\begin{tikzpicture}[scale=0.5]
\draw (-1,0) -| (0,-1);
\fill (0,0) circle (2pt);
\draw [loosely dotted, thick] (0,0) -- (1.0,1.0) node [anchor=south west] {$i+1$};
\end{tikzpicture}
}

In the case (1), furthermore it should be
\raisebox{-6ex}{
\begin{tikzpicture}[scale=0.5]
\draw (-1,1) -- (0,1) |- (1,0);
\fill (0,0) circle (2pt);
\fill (0,1) circle (2pt);
\draw [loosely dotted, thick] (0,0) -- (1.0,1.0) node [anchor=south west] {$i+1$};
\draw [loosely dotted, thick] (0,1) -- (1.0,2.0) node [anchor=south west] {$i$};
\end{tikzpicture}
}
since $\ka/\tau$ is a horizontal strip.
Besides, the case (1) should happen because the action of $s_{i+1}$ must add more than or equal to one box.

In fact the case (4) never happens since the action of $s_{i+1}$ does not remove boxes.

The case (3) is divided to 

(3-1):
\raisebox{-4ex}{
\begin{tikzpicture}[scale=0.5]
\draw (-1,1) -- (-1,0) -- (1,0);
\fill (0,0) circle (2pt);
\fill (-1,0) circle (2pt);
\draw [loosely dotted, thick] (0,0) -- (1.0,1.0) node [anchor=south west] {$i+1$};
\draw [loosely dotted, thick] (-1,0) -- (0,1) node [anchor=south west] {$i$};
\end{tikzpicture}
}
and
(3-2):
\raisebox{-4ex}{
\begin{tikzpicture}[scale=0.5]
\draw (-2,0) -- (-1,0) -- (1,0);
\fill (0,0) circle (2pt);
\fill (-1,0) circle (2pt);
\draw [loosely dotted, thick] (0,0) -- (1.0,1.0) node [anchor=south west] {$i+1$};
\draw [loosely dotted, thick] (-1,0) -- (0,1) node [anchor=south west] {$i$};
\end{tikzpicture}
}
.

The case (3-1) should happen since later the action of $s_i$ must add more than or equal to one box.

Thus we have a contradiction that there are both addable corners and removable corners of residue $i$ in this moment.

\section{Explicit description of $A_{\mu,\bla,q}$}\label{sec:apdx_A_la_q}
\begin{lemm}\label{A_la_q}
In the setting of Lemma \ref{g_mu_r_expand},
\[
A_{\mu,\bar\la,q} = 
\begin{cases}
	(-1)^{|\mu/\bar\la|} & (\text{if $\mu/\bar\la:$ vertical strip and $q=|\mu/\bar\la|+r_{\mu'\bar\la'}$}), \\
	0 & (\text{otherwise}).
\end{cases}
\]
\end{lemm}

\begin{proof}
We fix $\bar\la$, and set $f_{\mu}(t) := \sum_{q} A_{\mu,\bar\la,q}t^q \in \Z[t]$.
Then the definition of $A_{\mu,\bar\la,q}$ (in the statement of Lemma \ref{g_mu_r_expand}) is transformed into the recursion formula
\begin{align*}
	f_{\bar\la}(t) &= t^{r_{\bar\la\bar\la}}, \\
	f_{\mu}(t) &= -\sum_{\substack{\mu/\ka:\text{ h.s.} \\ \bar\la\subset\ka\subsetneq\mu}} t^{r_{\mu\mu}-r_{\mu\ka}}f_\ka(t) \ \ \text{for $\mu\neq \bar\la$},
\end{align*}
and the desired result becomes the condition
\[
f_{\mu}(t) = 
\begin{cases}
	(-1)^{|\mu/\bar\la|} t^{|\mu/\bar\la|+r_{\mu'\bar\la'}} & \text{(if $\mu/\bar\la$: vertical strip)} \\
	0 & \text{(otherwise)}
\end{cases}.
\]

We prove it by induction on $|\mu|$ (for $\mu$ satisfying $\bar\la\subset\mu\subset \Rbl$).
The base case $\mu=\bar\la$ is obvious by definition.

Then we assume $\bar\la\subsetneq\mu\subset\Rbl$.
First we consider the case where $\mu/\bar\la$ is a vertical strip.
In this case we put
\begin{eqnarray*}
	\{x_1,\ldots,x_s\} &:=& \{x\mid \bar\la'_x<\mu'_x\} \ (x_1 < \cdots < x_s), \\
	a_i &:=& \mu'_{x_i} - \bar\la'_{x_i}, \\
	b_i &:=& \bar\la'_{x_i-1} - \bar\la'_{x_i} (\ge a_i)\ \text{(if $x_1=1$ set $b_1=\infty$)},
\end{eqnarray*}
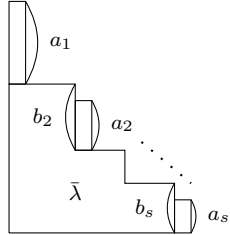
\[
\begin{tikzpicture}[scale=0.22]
\draw (0,0) -| (10,3) -| (7,5) -| (4,9) -| (0,0);
\node at (4,2.5) {\footnotesize $\bar\la$};
\draw [loosely dotted,thick] (8,5.5) -- (11,3);
\draw (0,9) rectangle (1,14);
\draw (1,9) to [out=60, in=-60] node[right=1pt]{\footnotesize $a_1$} (1,14);
\draw (4,5) rectangle (5,8);
\draw (5,5) to [out=60, in=-60] node[right=1pt]{\footnotesize $a_2$} (5,8);
\draw (4,5) to [out=120, in=-120] node[left=1pt]{\footnotesize $b_2$} (4,9);
\draw (10,0) rectangle (11,2);
\draw (11,0) to [out=60, in=-60] node[right=1pt]{\footnotesize $a_s$} (11,2);
\draw (10,0) to [out=120, in=-120] node[left=1pt]{\footnotesize $b_s$} (10,3);
\end{tikzpicture}
\]
and we denote by $\ka(c_1,\ldots,c_s)$ the partition defined by
$$
\ka(c_1,\ldots,c_s)'_x = 
\begin{cases}
	\bar\la'_x + c_i & \text{if $x=x_i$ for some $i$}\\
	\bar\la'_x & \text{otherwise}
\end{cases}
$$
for $0\le c_i \le a_i$ ($1\le i\le s$).
In particular $\ka(0,\ldots,0)=\bar\la$ and $\ka(a_1,\ldots,a_s)=\mu$.

Since $|\ka(c_1,\ldots,c_s)/\bar\la| = \sum_i c_i$ and $r_{\ka(c_1,\ldots,c_s)'\bar\la'} = r_{\bar\la\bar\la} - \#\{i\mid c_i=b_i\}$,
we have
\begin{align*}
	f_{\ka(c_1,\ldots,c_s)}(t) &= (-1)^{|\ka(c_1,\ldots,c_s)/\bar\la|} t^{|\ka(c_1,\ldots,c_s)/\bar\la|+r_{\ka(c_1,\ldots,c_s)'\bar\la'}} \\
											   &= (-1)^{\sum_i c_i}\ t^{r_{\bar\la\bar\la}+\sum_i (c_i-\DE{c_i=b_i})},
\end{align*}
for $0\le c_i \le a_i$ and $(c_1,\ldots,c_s)\neq(a_1,\ldots,a_s)$, by the induction hypothesis.

For $S\subset \{1,\ldots,s\}$, we set 
$\ka(S) = \ka(a_1-\de\left[1\in S\right],\ldots,a_s-\de\left[s\in S\right])$.
Then 
\[
	\begin{cases}
		\bar\la\subset\ka\subsetneq\mu\\
		\mu/\ka \text{: horizontal strip}
	\end{cases}
	\iff
	\ka = \ka(S) \text{\ for $\emptyset\neq\exists S \subset \{1,\ldots,s\}$}.
\]
Therefore,
\begin{align*}
	f_{\mu}(t) &= -\sum_{\substack{\mu/\ka\text{: h.s.} \\ \bar\la\subset\ka\subsetneq\mu}} t^{r_{\mu\mu}-r_{\mu\ka}}f_\ka(t) \\
       &= - \sum_{\substack{\emptyset\neq S \subset\{1,2,\ldots,s\} }} t^{r_{\mu\mu}-r_{\mu\ka(S)}}f_{\ka(S)}(t)\\
&=-t^{r_{\mu\mu}-r_{\mu\ka(\{1\})}}f_{\ka(\{1\})}(t) \\
&\phantom{=} -\sum_{\emptyset\neq T\subset\{2,\ldots,s\}} 
\bigg(\underbrace{t^{r_{\mu\mu}-r_{\mu\ka(T)}}f_{\ka(T)}(t)
+ t^{r_{\mu\mu}-r_{\mu\ka(\{1\}\cup T)}}f_{\ka(\{1\}\cup T)}(t)}_{(X)}\bigg).
\end{align*}

In fact it can be proved that $(X)=0$ by the following Claim 1 and Claim 2.

\vspace{2mm}
\noindent{\bf Claim 1.}
$$r_{\mu,\ka(\{1\}\cup T)} = r_{\mu,\ka(T)} - \DE{a_1<b_1} \ 
\text{(for all $T \subset \{2,\ldots,s\}$)}.$$

\noindent{\it Proof of Claim 1:}
Reduced to next lemma:

\begin{lemm}\label{p-p}
	Let $\gamma\subset\beta$ and 
	$y=(r,c)$ be an addable corner of $\gamma$.
	Put $\ti\ga=\ga\cup\{y\}$.
	Assume that $\ti\ga_1+l(\ti\ga)\le k+1$ and $y$ is $\beta$-nonblocked.
	Then
	\[
	r_{\beta\ti\gamma}-r_{\beta\gamma}=
	\begin{cases}
		0 & (\text{if $(r,c-1)$ is a $\beta$-nonblocked removable corner of $\ga$}),\\
		1 & (\text{otherwise}).
	\end{cases}
\]
\end{lemm}
\begin{proof}[Proof of Lemma \ref{p-p}]
Note that $r_{\beta\ti\ga}=\#\{\text{$\beta$-nonblocked $\ti\ga$-removable corners}\}$ since $\ti\ga\subset\Rbl$,
and the same equality holds for $r_{\beta\ga}$.

If $z$ is a $\beta$-blocked (resp. nonblocked) removable corner of $\ga$ other than $(r-1,c)$ or $(r,c-1)$, 
then $z$ is also $\beta$-blocked (resp. nonblocked) removable corner of $\ti\ga$, and vice versa.
Note that
\begin{itemize}
	\item $y=(r,c)$ is a $\beta$-nonblocked removable corner of $\ti\ga$, and not in $\ga$.
	\item $(r,c-1)$ is not a removable corner of $\ti\ga$.
	\item $(r-1,c)$ is not a removable corner of $\ti\ga$.
	Even if $(r-1,c)$ is a removable corner of $\ga$, 
	it is $\beta$-blocked.
\end{itemize}
Hence we conclude
\[
r_{\beta,\ti\ga}-r_{\beta,\ga}=
\begin{cases}
	0 & \text{(if $(r,c-1)$ is a $\beta$-nonblocked removable corner of $\ga$)}, \\
	1 & \text{(otherwise)}.
\end{cases}
\]
	\begin{center}
\begin{tabular}{c|cccc}
	\raisebox{5mm}{
	\begin{minipage}{.21\textwidth}
		{boundary of $\tilde\gamma$ around $y=(r,c)$} 
	\end{minipage}
	}&
\begin{tikzpicture}[scale=0.25]
	\draw (-1,1) -| (1,-1);
	\draw (0,1) |- (1,0);
	\draw (0.5,0.5) to [out=90,in=190] (2,2) node[right]{$y$};
	\draw[dotted,thick] (-1,1)--(-2,2);
	\draw[dotted,thick] (1,-1)--(2,-2);
\end{tikzpicture}
&
\begin{tikzpicture}[scale=0.25]
	\draw (-1,1) -| (1,0);
	\draw (0,1) |- (2,0);
	\draw (0.5,0.5) to [out=90,in=190] (2,2) node[right]{$y$};
	\draw[dotted,thick] (-1,1)--(-2,2);
	\draw[dotted,thick] (2,0)--(3,-1);
\end{tikzpicture}
&
\begin{tikzpicture}[scale=0.25]
	\draw (0,1) -| (1,-1);
	\draw (0,2) |- (1,0);
	\draw (0.5,0.5) to [out=90,in=190] (2,2) node[right]{$y$};
	\draw[dotted,thick] (0,2)--(-1,3);
	\draw[dotted,thick] (1,-1)--(2,-2);
\end{tikzpicture}
&
\begin{tikzpicture}[scale=0.25]
	\draw (0,1) -| (1,0);
	\draw (0,2) |- (2,0);
	\draw (0.5,0.5) to [out=90,in=190] (2,2) node[right]{$y$};
	\draw[dotted,thick] (0,2)--(-1,3);
	\draw[dotted,thick] (2,0)--(3,-1);
\end{tikzpicture}
\\[4pt]
$r_{\beta\ti\ga}-r_{\beta\ga}$ &
{\small $\DE{(r+1,c-1)\notin\beta}$} &
{\small $\DE{(r+1,c-1)\notin\beta}$} &
$0$ &
$0$
\end{tabular}
\end{center}

\end{proof}

\vspace{2mm}
\noindent{\bf Claim 2.}
$$
f_{\ka(\{1\}\cup T)}(t) 
	 = - f_{\ka(T)}(t)\cdot t^{-\DE{a_1<b_1}}
$$
for $\emptyset\neq T\subset\{2,\ldots,s\}$ 

\noindent{\it Proof of Claim 2:}
Put $a'_i=a_i-\de\left[i\in T\right]$, then
\begin{align*}
f_{\ka(\{1\}\cup T)}(t) &=
	(-1)^{|\mu/\bar\la|-|T|-1} t^{r_{\bar\la\bar\la}+(a_1-1)+\sum_{i>1} (a'_i-\DE{a'_i=b_i})} \\
	&= - (-1)^{|\mu/\bar\la|-|T|} t^{r_{\bar\la\bar\la}+(a_1-\DE{a_1=b_1})-\DE{a_1<b_1}+\sum_{i>1} (a'_i-\DE{a'_i=b_i})} \\
	&= - f_{\ka(T)}(t)\cdot t^{-\DE{a_1<b_1}}.
\end{align*}
\rightline{({\it End of the proof of Claim 2})}

Hence,
\begin{align*}
f_{\mu}(t) &= -t^{r_{\mu\mu}-r_{\mu\ka(\{1\})}}f_{\ka(\{1\})}(t) \\
				       &= -t^{\DE{a_1<b_1}} \cdot (-1)^{|\mu/\bar\la|-1}t^{r_{\bar\la\bar\la}+\sum_i(a_i-\DE{a_i=b_i})-\DE{a_1<b_1}} \\
	   &= (-1)^{|\mu/\bar\la|} t^{r_{\bar\la\bar\la}+\sum_i(a_i-\DE{a_i=b_i})}\\
	   &= (-1)^{|\mu/\bar\la|} t^{|\mu/\bar\la|+r_{\mu'\bar\la'}}.
\end{align*}
This completes the proof in the case where $\mu/\bar\la$ is a vertical strip.

Next we consider the case where $\mu/\bar\la$ is {\it not} a vertical strip.

We take the same $x_i$, $a_i$, $b_i$ ($1\le i\le s$) as above (in this case we have $a_i>b_i$ for some $i$),
and $\ka(c_1,\ldots,c_s)$
($0\le c_i \le a_i, 1\le i \le s$, so long as adding $c_i$ cells on top of the $x_i$-th column of $\bla$, for all $i$,
 yields a Young diagram of a partition)
and $\ka(S)$ ($S\subset\{1,2,\dots,s\}$ but bound by the same restriction).

Notice that $\ka(c_1,\ldots,c_s)/\bar\la$ is a vertical strip if and only if
$c_i\le b_i$ for all $i$.

Now we have
\begin{align*}
	f_{\mu}(t) &= -\sum_{\substack{\mu/\ka\text{: h.s.} \\ \bar\la\subset\ka \\ \ka\neq\mu}} t^{r_{\mu\mu}-r_{\mu\ka}}f_\ka(t). \\
\intertext{
	By the induction hypothesis, we have $f_{\ka}(t)=0$ unless $\ka/\bla$ is a vertical strip.
	Since $\mu/\ka$ must be a horizontal strip, $\ka$ must have the form $\ka(S)$.
	Therefore
}
	&= - \sum_{\substack{\emptyset\neq S \subset\{1,2,\ldots,s\} \\ a_i-\DE{i\in S} \le b_i \ (\forall i)}} 
	t^{r_{\mu\mu}-r_{\mu\ka(S)}}f_{\ka(S)}(t)\\
\intertext{
	If there exists some $i$ such that $a_i>b_i+1$, 
	then $f_\mu(t)=0$ since it is equal to an empty sum.
	So we assume that there is no such $i$.
	We set $U := \{i\in \{1,\ldots,s\}\mid a_i = b_i+1\} \neq \emptyset$.
	It is easily seen that $1\notin U$.
	Then
}
	&= - \sum_{\substack{U\subset S \subset\{1,2,\ldots,s\} }} 
	t^{r_{\mu\mu}-r_{\mu\ka(S)}}f_{\ka(S)}(t)\\
&= -\sum_{U\subset T\subset\{2,\ldots,s\}} 
\bigg(\underbrace{t^{r_{\mu\mu}-r_{\mu\ka(T)}}f_{\ka(T)}(t)
+ t^{r_{\mu\mu}-r_{\mu\ka(\{1\}\cup T)}}f_{\ka(\{1\}\cup T)}(t)}_{(X)}\bigg)\\
&= 0 \quad\text{(since $(X)=0$ by the same reason as the above case)}.
\end{align*}
\end{proof}

\noindent{\it Remark.}
By Lemma \ref{A_la_q}, Lemma \ref{g_mu_r_expand}(1), say, can be rewritten as:
\[
	\kks{\la} = \sum_{\substack{\mu \text{ s.t.}\\ \bar\la\subset\mu\subset \Rbl \\ \mu/\bar\la \text{: v.s.}}} (-1)^{|\mu/\bar\la|} \kks{\mu} \sum_{i\ge 0}
	\binom{(|\mu/\bar\la|+r_{\mu'\bar\la'})+i-1}{i} h_{r-|\mu/\bar\la|-i}.
\]

\end{document}